\documentclass[12pt]{amsart}
   \usepackage{amsmath,amsthm}
   \usepackage{amsfonts}   
   \usepackage{amssymb}    
   \usepackage{url}
   \usepackage{pstricks}
      \usepackage{pstricks,graphicx}
\usepackage{comment}
\advance \topmargin by -\headheight
\advance \topmargin by -\headsep
      
\evensidemargin \oddsidemargin
\usepackage[all]{xy}

\newtheorem{thm}{Theorem}[section]
\newtheorem{prop}[thm]{Proposition}
\newtheorem{lem}[thm]{Lemma}
\newtheorem{cor}[thm]{Corollary}

\newtheorem{defn}[thm]{Definition}

\newtheorem{prob}[thm]{Problem}

\theoremstyle{remark}
\newtheorem{rem}[thm]{Remark}
\newtheorem{ex}[thm]{Example}

\newcommand{\R}{\mathbb{ R}}

\def\co{\colon\thinspace}

\title{Approximating $C^0$-foliations by contact structures}
\author{Jonathan Bowden}
\address{Mathematisches Institut, Ludwig-Maximillians-Universit\"at, Theresienstr. 39, 80333 Munich, Germany}
\email{jonathan.bowden@math.lmu.de}
\date{\today}
\begin{document}
\maketitle
\begin{abstract}
We show that any co-orientable foliation of dimension two on a closed orientable $3$-manifold with continuous tangent plane field can be $C^0$-approximated by both positive and negative contact structures unless all leaves of the foliation are simply connected. 
As applications we deduce that the existence of a taut $C^0$-foliation implies the existence of universally tight contact structures in the same homotopy class of plane fields and that a closed $3$-manifold that admits a taut $C^0$-foliation of codimension-1 is not an $L$-space in the sense of Heegaard-Floer homology.


\end{abstract}

\section{Introduction}
An important breakthrough in low-dimensional contact topology in the 90's was Eliashberg and Thurston's discovery of a fundamental link between the theory of codimension-$1$ foliations and contact structures. This link is provided by the following approximation result:
\begin{thm}[Eliashberg-Thurston \cite{ETh}]\label{Eli_Th}
Let $\mathcal{F}$ be a foliation of class $C^2$ on a closed oriented $3$-manifold that is not the foliation by spheres on $S^2 \times S^1$. Then $T\mathcal{F}$ can be $C^0$-approximated by positive and negative contact structures.
\end{thm}
\noindent In this article we will consider extensions of this theorem to foliations of lower regularity. Eliashberg and Thurston already noted that Theorem \ref{Eli_Th} also holds for foliations that are smooth away from a finite collection of compact leaves. They also write (\cite{ETh}, p.\ 44)

\begin{quote}``However it is feasible that the result holds without any assumptions about the smoothness of the foliation.''
\end{quote}
As a matter of fact the statement of Theorem 2.9.1 on p.\ 44 in \cite{ETh} is for $C^0$-foliations but this is due to a (possibly Freudian) typographical slip. The main result of this article is that the statement quoted above is indeed correct. In order that approximation even makes sense for  $C^0$-foliations, we will only consider smooth leaved foliations that have continuous tangent distributions (cf.\ Section \ref{sec_fol} for precise definitions). So when we say that a contact structure approximates a $C^0$-foliation we will mean it approximates its tangent distribution $T\mathcal{F}$.
\begin{thm}\label{approx_cont}
Let $ \mathcal{F}$ be a $C^0$-foliation on a closed oriented $3$-manifold that is neither the foliation by spheres on $S^2 \times S^1$ nor a foliation by planes, in which case the manifold is $T^3$. Then $T\mathcal{F}$ can be $C^0$-approximated by positive and negative contact structures.
\end{thm}
Note that if $\mathcal{F}$ is a foliation by planes then after collapsing product regions the resulting foliation is topologically conjugate to a smooth foliation given by the kernel of some closed $1$-form by a result of Imanishi \cite{Ima}. This smooth foliation can then be approximated by contact structures as in Eliashberg and Thurston's proof of Theorem \ref{Eli_Th}. So in this case $\mathcal{F}$ is \emph{semi-conjugate} to a foliation that can be approximated by contact structures. The main application of Theorem \ref{approx_cont} is to construct tight contact structures from taut foliations.
\begin{thm}
Let $ \mathcal{F}$ be a taut $C^0$-foliation on a closed oriented $3$-manifold that is not the foliation by spheres on $S^2 \times S^1$. Then there are both positive and negative contact structures $\xi_+$ and $\xi_-$ that are symplectically semi-fillable, universally tight and homotopic as plane fields to $T\mathcal{F}$.
\end{thm}
\noindent A similar statement holds for Reebless foliations (Theorem \ref{Reebless}). A further consequence is that $L$-spaces in the sense of Heegaard-Floer theory do not admit any taut foliations (Corollary \ref{L_Space}). We remark that a weakened version of Theorem \ref{approx_cont} for special classes of taut foliations has been obtained in \cite{KR} using rather different methods. Moreover, Kazez and Roberts have refined their methods to give an independent proof of Theorem \ref{approx_cont} for general taut foliations \cite{KR2} without any assumptions other than that the manifold is not $S^2\times S^1$ with the product foliation. In particular, the hypothesis that the foliation is not a (non-minimal) foliation by planes is purely technical and can be removed. 

To illustrate the respective similarities and differences between the $C^2$-case and its $C^0$-counterpart we summarise the steps involved in the proof of Theorem \ref{Eli_Th} and \ref{approx_cont} respectively:
\subsection*{$C^2$-case} Here there are 3 main steps:
\begin{enumerate}
\item After an initial perturbation we can assume that there are finitely many compact leaves.
\item One then produces contact regions near minimal sets using holonomy: For closed leaves this is done by thickening and inserting suspension foliations, for exceptional minimal sets one appeals to Sacksteder's theorem and for minimal foliations one uses a special argument depending on whether there is holonomy or not.
\item The contactness can then be transported around $M$ since any leaf accumulates on some minimal set.
\end{enumerate}
\subsection*{$C^0$-case}
The basic observation, which has been utilised to great effect by Colin \cite{Col} and later by Vogel \cite{Vog}, is that a contact structure is essentially determined by its restriction to the $2$-skeleton of a suitable polyhedral decomposition. More precisely, any smooth plane field defined near the $2$-skeleton of a polyhedral decomposition such that the holonomies are negative on the (oriented) boundaries of all $3$-cells can be extended to a (positive) contact structure over the interiors of all polyhedra. This means that it suffices to find a plane field near the $2$-skeleton with this property that is close to the original foliation. The steps in the $C^0$-case are then as follows:
\begin{enumerate}
\item[($1'$)] Assume that the number of closed leaves is finite as in the $C^2$-case above.
\item[($2'$)]  Here we treat all cases equally: one introduces holonomy by thickening leaves (compact or not) and inserting suitable suspension foliations. After this one then produces positive/negative contact regions near these holonomy curves.
\item[($3'$)]  The contactness is then transported via \emph{disjoint} thickenings of embedded arcs tangent to the plane field so that the resulting plane field induces the correct holonomies on the $2$-skeleton of some polyhedral decomposition outside small neighbourhoods of the holonomy curves. (Here care is needed so that holonomy makes sense cf.\ Defintion \ref{def_confol} ff.)
\item[($4'$)]  One then fills in the $3$-cells by contact structures away from neighbourhoods of the holonomy curves from step ($2'$) and finally extends the contact structure over (slightly enlarged) neighbourhoods of these holonomy curves.
\end{enumerate}
For the second step in the $C^0$-case above, one needs to carefully analyse the exceptional minimal sets of a general $C^0$-foliation. These have two important properties: they are finite and each contains a non-simply connected leaf unless $M$ is $T^3$ or $S^2 \times S^1$ and the foliation is very special, in that \emph{all} leaves are simply connected. Furthermore, the reason that the final step is divided into two parts is that one first needs to smoothen before one can transport the contactness. For foliations this can only be carried out near disjoint leafwise curves and in general there are $C^0$-foliations that cannot globally be $C^0$-approximated by ones of class $C^2$ -- for example there are restrictions given by Kopell's Lemma. 

An important technical point here is that a continuous plane field does not induce well defined holonomy maps as it may not be uniquely integrable. However, as the original plane field was tangent to a foliation and all modifications are performed in smooth regions, this means that the plane field on the boundary of $3$-cells can be integrated to a (possibly non-unique)  foliation for which holonomy is then well-defined. Moreover, this holonomy is canonical once one fixes a foliation tangent to the original $C^0$-plane field. These facts mean that the bulk of the technical work involved in the proof is contained in this last step and one has to be both very careful how, and in which order, one smoothens the underlying plane field.

We also point out that even in the $C^2$-case our approach to approximating foliations by contact structures yields 2 simplifications. The first is that due to the fact that as the contactness is transported along \emph{disjoint} neighbourhoods, these neighbourhoods do not interact meaning so that the corresponding discussion in Petronio \cite{Pet} is no longer needed. Secondly we show that a foliation without holonomy can be approximated by contact structures directly without first being approximated by surface fibrations (unless the foliation is by planes).


Now that one has an approximation result for $C^0$-foliations it is natural to ask how much of the theory extends to this more general setting. For example Vogel \cite{Vog} has recently shown that the isotopy class of an approximating contact structure is unique for $C^2$-foliations without torus leaves, with a short list of exceptions. On the other hand Vogel's proof uses the $C^2$-assumption in an essential way and it seems unlikely that this result should extend to the case of $C^0$-foliations. Several examples related to this question are discussed in Section \ref{sec_discussion}. In a similar vein the author \cite{Bow} recently showed that any contact structure that is sufficiently close to a Reebless foliation of class $C^2$ is universally tight. However, this proof again uses the $C^2$-condition in an essential way. Moreover, the phenomena of phantom Reeb tori highlighted by Kazez and Roberts \cite{KR3} show that the corresponding result for $C^0$-foliations is false. However, the result ought to hold in the case of $C^1$-foliations for which phantom Reeb tori do not appear.

\subsection*{Outline of Paper:} In Section \ref{sec_fol} we review the necessary definitions and basic results from the theory of contact structures and foliations and we also give a (working) definition of a $C^0$-confoliation. Section \ref{sec_smooth} contains technical results about smoothing near the $1$-skeleton of a suitably transverse triangulation as well as near leafwise arcs. Section  \ref{nice_section}  introduces the notion of nice coordinates near closed leafwise embedded curves and in Section \ref{contact_regions} we review how to produce contact structure near curves with attracting holonomy. In Section \ref{jiggling_sec} we discuss jiggling and polyhedral decompositions  \`a  la  Colin and we also prove the main technical results needed to fill in contact structures over the interior of $3$-cells in a controlled way. Section \ref{sec_ribbons} reviews Vogel's theory of ribbons. Finally in Section \ref{sec_proof} we prove Theorem \ref{approx_cont} and in Section \ref{sec_discussion} we discuss various applications and examples.

\subsection*{Acknowledgments:} 
The author was partially supported by DFG Grant BO4423/1-1. We thank the referee for helpful comments.
\medskip

\noindent {\bf Conventions:} Throughout $M$ will denote a smooth, oriented, connected $3$-manifold and this manifold will be closed unless stated otherwise. All foliations are of codimension-$1$ and all measurements will be taken with respect to a {\em fixed} background metric. For any subset $A$ of a manifold $\mathcal{O}p(A)$ will denote some (unspecified) neighbourhood of $A$.

\section{$C^0$-Foliations and minimal sets}\label{sec_fol}
We first recall the definition of a foliation, paying special attention to the various regularity assumptions one needs when considering $C^0$-foliations. We begin with the most general definition possible:
\begin{defn}
A topological codimension-$1$ foliation $\mathcal{F}$ on a $3$-manifold $M$ is a decomposition of the manifold into (topologically) embedded surfaces called {\bf leaves} that is given by a topological atlas whose transition functions $\varphi_i \circ \varphi_j^{-1}$ preserve the planes $\R^2 \times \{pt\}$. If the transition maps are smooth on the $\R^2$-slices, then the foliation is of class $C^{0,\infty}$.

Finally if $\mathcal{F}$ is in addition tangent to a $C^0$-plane field, then we say that it is of class $C^{0,\infty+}$.
\end{defn}
Throughout this article we will also assume that all foliations are {\bf cooriented}, so that the tangent plane field $T\mathcal{F}$  in the case of a foliation of class $C^{0,\infty+}$ can be defined as the kernel of a (continuous) non-vanishing $1$-form.

It is fairly easy to see that if the transition maps are of class at least $C^1$ on leaves and the foliation is transverse to a smooth flow, then one can find a foliated atlas that is leafwise smooth (cf.\ \cite{CC} Corollary 5.15). However, the fact that the foliation is tangent to a continuous plane field is \emph{a priori} much stronger as it implies that the leaves are ($C^1$-)immersed and that (locally) these immersions vary continuously in the $C^1$-topology. 

In higher dimensions, it is not clear how all these conditions fit together, but due to the following straightening theorem of Calegari, we know that any topological foliation by surfaces is  topologically isotopic to one of class $C^{0,\infty+}$. 
\begin{thm}[Calegari \cite{Cal}]\label{Calegari}
Let $\mathcal{F}$ be a topological foliation by surfaces on a $3$-manifold. Then $\mathcal{F}$ is topologically isotopic to one of class $C^{0,\infty+}$. Moreover, one can assume that the leaves are $C^{\infty}$-immersed.
\end{thm}
\begin{rem}
It appears that there might be some difference between having $C^1$-immersed or $C^{\infty}$-immersed leaves. However, it is not hard to see that a foliation with $C^1$-immersed leaves and continuous tangent plane field can be approximated (in the tangential sense) by one whose leaves are $C^{\infty}$-immersed (cf.\ Remark \ref{smoothing_attempt} below).
\end{rem}

\noindent {\bf Convention:} From now on a $C^0$-foliation will mean a cooriented foliation of class $C^{0,\infty+}$.

\subsection*{Approximating $C^0$-foliations by other foliations}\label{fol_norm_etc}
Throughout this article we will approximate $C^0$-foliations by other foliations with nicer properties (eg.\ smoothness, nice local product structures...). This will be done locally in a smooth coordinate patch $U \subseteq M$ where the foliation is given as the graphs of a family of functions (with parameter $z$ in an interval $I$)
$$f_z(x,y)\co D^2 \longrightarrow \R $$
and the $C^{0,\infty+}$ condition means that we can assume that the partial derivatives $\frac{\partial   f_z}{\partial x},\frac{\partial   f_z}{\partial y}$ are continuous in the smooth coordinates $(x,y,z)$. Note that the parameter $z$ can be taken as the intersection of a leaf with the $z$-axis and the property of being a foliation is that for fixed $(x_0,y_0) \in D^2$ the function $f_z(x_0,y_0)$ is continuous and strictly monotone in $z$. Two foliations $\mathcal{F},\mathcal{G}$ are close on $U$ if for the associated family of functions
$$\left\| f_z(x,y)  - g_z(x,y) \right\|_{C^0_{Fol}} = \left\|\left(f_z(x,y), \frac{\partial   f_z}{\partial x} \wedge \frac{\partial   f_z}{\partial y}\right) -  \left(g_z(x,y), \frac{\partial   g_z}{\partial x} \wedge \frac{\partial   g_z}{\partial y}\right)  \right\|_0 < \epsilon.$$
Here the wedge product of two vectors denotes the subspace spanned by them in the (oriented) Grassmannian. This is then equivalent to the $C^0$-closeness of the associated foliations and hence we call it the {\bf foliated $C^0$-norm}. In particular, if the functions are $C^1$-close on $U$ in the sense that the partial $C^1$-norm
$$\left\| f_z(x,y)  - g_z(x,y) \right\|_{C^1_{part}} :=\left\|\left(f_z(x,y), \frac{\partial   f_z}{\partial x}, \frac{\partial   f_z}{\partial y}\right) -  \left(g_z(x,y), \frac{\partial   g_z}{\partial x}, \frac{\partial   g_z}{\partial y}\right)  \right\|_0 < \epsilon,$$
then so are the resulting foliations. 
We then have the following cutting-off lemma, that will be essential for approximating $C^0$-foliations on foliated charts:
\begin{lem}[Cutting-off Lemma]\label{loc_principle}
Let $f_z(x,y)\co D^2 \longrightarrow \R$ be associated to a local parametrisation of a $C^0$-foliation $\mathcal{F}$ on some $U \subseteq M $. Suppose that $$f^{n}_z(x,y) \longrightarrow f_z(x,y)$$ in the partial $C^1$-norm and that each $f^{n}_z(x,y)$ also determines a foliation (that is they are monotone in $z$). Then there are foliations $\mathcal{F}_{n} $ converging to $ \mathcal{F}$ in the $C_{Fol}^0$-sense that agree with the foliation determined by $f^{n}_z(x,y)$ on any $V \subseteq U$ with compact closure in $ int(U)$. The result also holds if the initial convergence is only in the $C^0_{Fol}$-sense.
\end{lem}
\begin{proof}
Let $\rho\co U \longrightarrow [0,1]$ be a bump function which has support on $U$ and is identically $1$ on $\overline{V}$. Then set $\mathcal{F}_{n}$ to be the foliation given by
$$g_z(x,y) = (1-\rho(x,y,z)) f_z(x,y) + \rho(x,y,z) f^{n}_z(x,y)$$
on $U$ and by $\mathcal{F}$ outside of $U$. This then gives the desired approximating family.
\end{proof}
\subsection*{Confoliations and contact structures}
In the case that a foliation is smooth, meaning that its defining atlas can be chosen so that all transition maps are smooth, then any smooth defining form $\alpha$ for its tangent distribution $T\mathcal{F}$ satisfies
$$\alpha \wedge d \alpha \equiv 0.$$
Conversely, this is equivalent to the existence of a foliation tangent to a given plane field. On the other hand a smooth plane field $\xi = Ker(\alpha)$ is  completely non-integrable or a (positive) {\bf contact structure}, if
$$\alpha \wedge d \alpha > 0.$$
If $\alpha$ satisfies the weaker inequality $\alpha \wedge d \alpha \ge 0$, then $\xi$ is called a (positive) {\bf confoliation}. A negative contact structure resp.\ negative confoliation is one for which the above inequalities are reversed.

\subsection*{Tautness and Reeblessness}
A codimension-$1$ foliation $\mathcal{F}$ on $M$ is {\bf (everywhere) taut} if every point is contained in a smoothly embedded closed curve transverse to $\mathcal{F}$. It is usually customary to require only that each leaf is cut by a closed transversal and in the case of $C^1$-foliations these notions are equivalent, but for $C^0$-foliations it may be weaker (cf. Kazez and Roberts \cite{KR3}). In order to distinguish these definitions (following \cite{KR3}) we will call the latter class of foliations  {\bf smoothly taut}. There are several equivalent conditions for tautness and the following result of Sullivan \cite{Sull} is particularly useful for understanding the nature of contact structures that are close to taut foliations (see also \cite{CC} Proposition 10.4.1).
\begin{lem}
A $C^0$-foliation $\mathcal{F}$ is taut if and only if there is a closed non-vanishing $2$-form so that $\omega|_{T\mathcal{F}} >0$.
\end{lem}
\begin{rem}
Note that in the reference \cite{CC}, the lemma above is stated for smoothly taut foliations, but this is incorrect due to the examples exhibited in \cite{KR3}. However as any smoothly taut foliation can be $C^0$-approximated by one that is everywhere taut it is true for a \emph{generic} taut foliation.
\end{rem}

A slightly weaker condition than tautness is that a foliation $\mathcal{F}$ has no Reeb components, i.e.\ there are no torus leaves bounding solid tori whose interiors are foliated by planes. In this case the foliation is called {\bf  Reebless}. The condition of Reeblessness puts restrictions on the topology of the leaves of the foliation as well as the underlying manifold itself. For the $C^2$-case this is due to Novikov \cite{Nov} and was extended to $C^0$-foliations by Solodov \cite{Sol}. Note that in the case that the foliation is of class $C^{0,\infty+}$ the original proof of Novikov in fact generalises in a more or less direct fashion.
\begin{thm}[Novikov, Solodov]\label{Novikov}
Let $\mathcal{F}$ be a Reebless $C^0$-foliation on a $3$-manifold which is not the product foliation on $S^2 \times S^1$. Then the inclusion of any leaf $L \hookrightarrow M$ is $\pi_1$-injective, $\pi_2(M) = 0$ and all transverse loops are essential in $\pi_1(M)$. In particular, $\pi_1(M)$ is infinite.
\end{thm}

\subsection*{Vanishing cycles and closed leaves}
Novikov  \cite{Nov} has shown that the existence of a Reeb component is equivalent to that of a vanishing cycle:
\begin{defn}
An embedded curve $\sigma_0\co S^1 \longrightarrow L_0 $ lying on a leaf of a foliation $\mathcal{F}$ of a manifold $M$ is called a {\bf vanishing cycle} if there is an embedding $\sigma_t \co S^1 \times [0,\epsilon] \longrightarrow M$ so that for fixed $t$ the image $\gamma_t = \sigma_t(S^1)$ lies on a leaf $L_t$ and $\gamma_t$ is contractible in $L_t$ but $\sigma_0$ is not contractible in $L_0$.
\end{defn}
\noindent We will only need a slightly weaker version of Novikov's result about vanishing cycles and Reeb components, which is most easily seen using Sullivan's theory of foliation cycles. For completeness we briefly recall the argument (cf.\ \cite{CC} pp.\ 259-61).
\begin{prop}[Novikov]\label{nov_closed}
Let $\sigma_0 \co S^1 \longrightarrow M$ be a vanishing cycle for a foliation $\mathcal{F}$. Then $\gamma_0 = \sigma_0(S^1)$ lies on a closed leaf $L_0$ of $\mathcal{F}$ (which is necessarily a torus). 
\end{prop}
\begin{proof}
Let $D_t$ denote the disc bounding $\sigma_t$ for $t >0$. Define currents $c_t \in \left(\Omega_2(M)\right)' = \mathcal{D}_2'$ as follows:
$$c_{t}(\omega) = \frac{1}{Area(D_{t})}\int_{D_{t}} \omega,$$
where the area is measured with respect to a fixed Riemannian metric. Then since the lengths of $\partial D_t$ are bounded and $Area(D_t)$ is unbounded (\cite{CC} p. 260), there is a subsequence $t_n \nearrow 1$ so that $c_{t_n} \longrightarrow c_{\infty}$ converges in the weak sense. Moreover, $c_{\infty}$ is a (non-trivial) foliation cycle in the sense of Sullivan  (\cite{CC} Theorem 10.2.22). In fact $c_{\infty}$ is exact, meaning that it is trival on all closed forms. To see this note that the disc $D_t$ is homologous to a disc $D'_t = D_{t_0} \cup A_{t_0}$, where $A_{t_0}$ is the annulus between $\sigma_{t_0}$ and $\sigma_t$. In particular, $Area(D'_t)$ is bounded, so that
$$\frac{1}{Area(D_t)}\int_{D_t} \omega_{cl} = \frac{1}{Area(D_t)}\int_{D'_t} \omega_{cl}\longrightarrow 0$$
for any closed $2$-form $\omega_{cl}$. Thus $c_{\infty}$ must have support on a compact leaf (\cite{Sull}, Proposition II.16), which is a barrier in the sense of Novikov in that it admits no closed transversal. By an argument of Goodman such a leaf must be a torus (cf.\ \cite{CC} Theorem 6.3.5).
\end{proof}

\subsection*{Minimal sets}
An important step in proving that a foliation can be approximated by contact structures is to understand the structure of its minimal sets:
\begin{defn}
A subset $M_*$ in a foliated manifold is called {\bf saturated}, if for any point $x \in M_*$ the leaf $L_x$ through $x$ is also contained in $M_*$.

A non-empty closed saturated subset $M_*$ is called {\bf minimal} if it contains no smaller non-empty closed saturated subsets.
\end{defn}
There are three possibilities for a minimal set. Either
\begin{itemize}
\item $M_*$ is a compact leaf,
\item $M_* = M$ in which case the foliation $\mathcal{F}$ is called {\bf minimal},
\item if $M_*$ is not a compact leaf and $M_* \neq M$, then $M_*$ is called an {\bf exceptional minimal set}.
\end{itemize}
Note that the closure of any leaf $\overline{L}$ is a saturated subset and by Zorn's lemma contains at least one minimal set. Note also that an exceptional minimal set cannot be contained in an (injectively leafwise immersed) foliated product $L \times [0,\epsilon]$, where $L$ may be non-compact, since any closed saturated subset has a bottom-most leaf that cannot accumulate on any other leaf in the product. This then contradicts minimality, since every leaf $L_* \subset M_*$ in a minimal set is dense in the minimal set.

In order to manufacture holonomy it will be important to be able to find non-simply connected leaves in minimal sets. Fortunately minimal sets consisting entirely of planes are well understood (see also \cite{Li}). The following result was observed by Gabai, who reduced it to a result of Imanishi.
\begin{lem}[Gabai \cite{Gab}, Imanishi \cite{Ima}]\label{Gabai_planes}
Let $\mathcal{F}$ be a Reebless $C^0$-foliation on a manifold $M$ that has a minimal set all of whose leaves are planes. Then $M = T^3$ and  $\mathcal{F}$ itself is a foliation by planes, which is semi-conjugate to a smooth foliation.
\end{lem}
\noindent Since the formulation given above is not taken directly from either of the references, we provide some explanation. Gabai \cite{Gab} considered essential laminations by planes. The complement of such a lamination consists of product regions that can be filled in to obtain a $C^0$-foliation by planes. In particular, any exceptional minimal set of a Reebless foliation $\mathcal{F}$, which is an essential lamination more or less by definition, is such that its complement consists of product regions $\mathbb{R}^2 \times [0,\epsilon]$. These regions are then foliated by planes and, in particular, $\mathcal{F}$ itself must be a foliation by planes. Generalising a classical result of Rosenberg for $C^2$-foliations, Imanishi \cite{Ima} showed that such a foliation can only occur on $T^3$ and is semi-conjugate to a linear foliation.

 In fact Lemma \ref{Gabai_planes} is also true without the assumption that the foliation is Reebless.
\begin{lem}\label{lem_ima}
Let $\mathcal{F}$ be a $C^0$-foliation on a manifold $M$ that has a minimal set all of whose leaves are planes. Then $M = T^3$, and $\mathcal{F}$ itself is a foliation by planes, which is semi-conjugate to a smooth foliation.
\end{lem}
\begin{proof}
Using the fact that the set of compact leaves is compact, one can decompose $M$ into a union of two possibly disconnected codimension-0 submanifolds $M = M_{Ator} \cup M_{Tor}$ with toroidal boundary consisting of leaves (including the boundaries of Reeb components), so that $\mathcal{F}$ is without torus leaves on the interior of $M_{Ator}$ and $M_{Tor}$ is a union of foliated $I$-bundles over $T^2$. 

Now any minimal set $M_*$ all of whose leaves are planes must be contained in the \emph{interior} of a connected component $C$ of $M_{Ator}$. The boundary of $C$ is then a union of tori $T_1,\cdots T_k$. We then form the double $2C$ of $C$ and consider the foliation $\overline{\mathcal{F}}$ on $2C$ given by gluing the restriction of $\mathcal{F}$ to $C$ along the boundary torus leaves. This foliation contains $M_*$ as a minimal set and its toral leaves correspond precisely to $\partial C$. First note that none of these tori can bound Reeb components. For in this case $C$ would be a solid torus and $\mathcal{F}$ would be a Reeb component, whose only (non-trivial) minimal set is the boundary leaf, contradicting the fact that $C$ contains an exceptional minimal set. In particular, each $T_i$ is incompressible and is contained in a complementary regions of $M_*$. 

But as $M_*$ is then an essential lamination by planes each complementary region of $M_*$ is homeomorphic to $\mathbb{R}^2 \times (0,1)$ contradicting the fact that each $T_i$ is incompressible. It follows that the original foliation was without torus leaves, and hence Reebless, so we can apply Lemma \ref{Gabai_planes}.
\end{proof}

In order to make sure that our approximation process stops we will need to know that the number of exceptional minimal sets of a foliation is finite. We have the following, which is an immediate consequence of Sacksteder's Theorem in the $C^2$-case (cf.\ \cite{CC} Theorem 8.3.2). In the $C^0$-case the argument generalises quite easily using the Dippolito's notion of an Octopus Decomposition associated to a saturated open subset.
\begin{lem}\label{finite_exceptional}
The number of exceptional minimal sets of a $C^0$-foliation on a closed manifold is finite.
\end{lem}
\begin{proof}
Assume not and let $\{X_n\}$ be an infinite sequence of distinct exceptional minimal sets. Then as the set of (non-empty) compact subsets of a compact space is itself compact, some subsequence converges to a closed saturated subset $X_{\infty}$. We let $X_{*}$ be a minimal set in $X_{\infty}$. The complement of $X_{*}$ in $M$ consists of product regions and finitely many non-product components $C_1,\ldots, C_K$ whose metric completions admit Octopus Decompositions (cf.\ \cite{CC} p.\ 130 ff.). Note that no minimal set can lie in a product region, since the bottom-most leaf of such a closed saturated set cannot accumulate on any other leaf contained in the product region. Without loss of generality all $X_n$ lie in one non-product component $C_*$ whose metric completion has an Octopus Decomposition
$$\widehat{C}_* = K \cup B_1 \cup \cdots \cup B_p$$
where $K$ is compact and has boundary that decomposes into a tangential and transverse part and $B_i$ is a product region which intersects $K$ in an annulus and has tangential boundary consisting of subsets of leaves of $X_{*}$. Let $x_{*} \in X_{*}$ lie on a (semi-proper) border leaf $L_{b}$ of  $C_*$ (cf.\ \cite{CC}, p.\ 133 ff.). Then there is a sequence of points $x_n \in X_n$ converging to $ x_{*}$ and without loss of generality we can assume that $X_n \neq X_{*}$. We then choose a small compact transversal $T \cong [-\epsilon,\epsilon]$ at $x_{*}$. The nearest points of $T \cap X_{n}$ to $x_{*}$ give a sequence of points $y_n$ that are fixed by all holonomy maps of $L_{b}$, when these holonomies are defined so that the leaf $L_b$ is \emph{semi-stable}. By passing to a subsequence and possibly flipping the interval, we can assume that $y_n \in (0,\epsilon]$. Dippolito's Semi-Stability Theorem (\cite{CC} Theorem 5.3.4) then implies that all but finitely many $X_n$ are contained in a product region $L_{b} \times [0,\epsilon]$ one of whose boundary components is $L_b$ itself. But no exceptional minimal set can be contained in a product region. This contradiction shows that the number of exceptional minimal sets must be finite.
\end{proof}
\subsection*{Blowing up leaves} 
We next recall how to blow up leaves and insert holonomy, which will allow us to sidestep the use of Sacksteder's Theorem to approximate general $C^0$-foliations by contact structures. This construction goes back to Denjoy for flows and was extended to foliations of codimension-$1$ by Dippolito \cite{Dip}. The basic idea is to replace an arbitrary (non-compact) leaf $L$ by a saturated product. This can be done in a local foliation chart $D^2 \times [0,1]$ by blowing up each disc $D^2 \times \{p\} \subseteq L$ to a sufficiently small product. One then patches together these local blow-ups using the foliated atlas associated to the foliation. By making these local patching alterations via local leafwise isotopies one can also assume that the resulting blow-up is again of class $C^{0,\infty+}$ and that it is also $C^0$-close to the original foliation. This construction is carried out in detail in (\cite{Dip}, pp.\ 435--6) where it is referred to as ``implantation''.
\begin{lem}[Blowing up leaves \cite{Dip}]
Let $\mathcal{F}$ be a $C^0$-foliation and let $L$ be a leaf. Then we can blow up $\mathcal{F}$ along $L$ to replace it by a product of leaves $L \times [0,\epsilon]$. The resulting foliation can be assumed to be $C^0$-close to $\mathcal{F}$.
\end{lem}
\noindent Once we have a product region, then we can glue in any suspension foliation we wish.
\begin{lem}[Inserting holonomy]
Let $\mathcal{F}$ be a $C^0$-foliation containing a leaf $L$. Let $\mathcal{F}_{susp}$ be any suspension foliation on $ L \times [0,\epsilon]$ so that the boundary components are leaves. Then we can blow up along $L$ and replace the product foliation by a foliation conjugate to $\mathcal{F}_{susp}$. We can further assume that resulting foliation is $C^0$-close to $\mathcal{F}$ by squashing the product region to become arbitrarily thin.
\end{lem}
Both of these lemmas have converses in the sense that if one can find a (leafwise immersed) saturated product bundle $L \times [0,\epsilon]$ in a foliated manifold so that $L \times \{0\}$ and $L \times \{\epsilon \}$ are leaves, then this product region can be collapsed to obtain a foliation with less product regions. This process is called ``explantation'' in \cite{Dip}. Moreover, this collapsing is achieved by a homotopy of maps $h_t \co M \longrightarrow M$ which are diffeomorphisms for $t \in (0,1]$, $h_1 = \textrm{id}$ and $h_0$ is a map that collapses the product region $L \times [0,\epsilon]$ to $L$. Dippolito's description of explantation (\cite{Dip}, Theorem 7) is quite involved as he wishes to make the collapsing smooth for $t >0$, in order that certain invariants of the foliation remain constant. It is however much easier to construct an explantation through maps that are only \emph{leafwise} smooth: simply make each collapse in a chart and then patch them together on overlaps using a partition of unity (which clearly preserves leafwise smoothness).

As soon as a leaf $L$ is not simply connected one can blow up and insert holonomy along any (embedded) homotopically non-trivial curve.
\begin{lem}\label{hol_insertion}
Let $L$ be a leaf of a foliation and let $\gamma \subseteq L$ be an embedded closed curve that does not bound a disc. Then we may blow up $L$ to $L \times [0,\epsilon]$ and insert a suspension foliation so that the holonomy around $\gamma \subseteq L \times\{0\}$ has no fixed points in $(0,\epsilon)$.
\end{lem}
\begin{proof}
If $\gamma$ is homologically non-trivial then there is a homomorphism $\rho_{\gamma} \co \pi_1(L) \longrightarrow \mathbb{Z}$ so that $\gamma$ is mapped to a generator. We then choose any diffeomorphism $f$ of $[0,\epsilon]$ without fixed points in the interior which gives the desired suspension foliation via $\rho_{\gamma}$.

Assume now that $\gamma$ is homologically trivial and consider the two components $C_0,C_1$ of $L \setminus \gamma$, one of which must be compact. If the other component is non-compact then we can apply the argument above to obtain a representation so that the image of $\gamma$ is arbitrary on the non-compact component. Thus it suffices to consider the case that both components are compact so that the leaf $L$ is itself compact. By collapsing handles we can assume that the leaf has genus two and that $\gamma$ is the central curve. We now consider a diffeomorphism $f \in \textrm{Diff}_+([0,\epsilon])$ without interior fixed points which can be written as a commutator
$$f = ghg^{-1}h^{-1}.$$
Such a diffeomorphism can be constructed as follows: identify $(0,\epsilon)$ with $\mathbb{R} \subseteq S^1 = \mathbb{R} \cup\{\infty\}$ and consider the affine action
$$g(x) = 2x \ , \ h(x) = x+1.$$
In fact, in this case we have $ ghg^{-1}h^{-1} = h = f.$ By taking such representations on the components of $L \setminus \gamma$, we then obtain a representation $\pi_1(L ) \longrightarrow \textrm{Diff}_+([0,\epsilon])$ with $\rho_{f}(\gamma) = f$ giving the desired suspension foliation.
\end{proof}

\subsection*{$C^0$-Confoliations and holonomy}\label{holonomy}
It is a fundamental observation going back to Eliashberg and Thurston \cite{ETh} that the holonomies induced by a smooth (positive) confoliation must be non-positive, at least on a small scale. For the same reason the holonomies of a $C^0$-confoliation $\xi$, which we will now define precisely, are also non-positive. 
\begin{defn}\label{def_confol}
A plane field $\xi$ is a {\bf $C^0$-confoliation} if for each point $p \in M$ there is an open neighbourhood $U$ of $p$ so that $\xi$ restricted to $U$ is either tangent to the restriction of a given $C^0$-foliation $\mathcal{F}$ defined on all of $M$ or is a smooth confoliation. 
\end{defn}
\begin{rem}\label{rem_hol}
Note that the assumption that $\xi$ is tangent to a \emph{globally} defined foliation where it is not smooth is essential for our analysis below, as it allows one to define holonomies on subsets of $M$ in a coherent manner.
\end{rem}
\noindent In order to analyse the local holonomies of a $C^0$-confoliation we suppose that $U \cong D^2 \times [0,1]$ is a subset of $M$ so that the intervals $\{pt\} \times [0,1]$ are transverse to a plane field $\xi$. We then consider the plane field $\xi$ as a connection, so that we can define the parallel transport of a curve in $D^2$. In general, the tangent plane field may \emph{not be uniquely integrable} on subsets where it is tangent to a foliation $\mathcal{F}$ and is not smooth so that such parallel transport cannot be defined unambiguously by merely lifting curves. In order to circumvent this we take (the unique) lifts that are tangent to leaves of $\mathcal{F}$ in regions where $\xi$ is not smooth. Thus the key point to defining holonomies is that the characteristic line field imprinted on $\partial D^2 \times [0,1]$ can be \emph{coherently integrated} to a \textbf{characteristic foliation} in the sense that is the tangent plane field of a $C^0$-foliation of dimension $1$. Note that such a foliation need not be unique in general (cf.\ \cite{BonFr}).

The holonomy $h_{\xi}$ along any smooth embedded curve $\gamma \co [0,1] \longrightarrow D^2$ (when defined) is determined by considering the $\xi$-lift $\tilde{\gamma}_x$ of $\gamma$ in the sense described above starting at some point $x \in [0,1]$ and setting
$$h_{\xi}(x) = \tilde{\gamma}_x(1).$$
We claim that $h_{\xi}(x) \le x$, when $\xi$ is a $C^0$-confoliation. Moreover, if this lift passes through a region where $\xi$ is contact then $\tilde{\gamma}(1) < \tilde{\gamma}(0)$ and hence $h_{\xi}(x) < x$. 

To see this one chops the region bounded by $\gamma$ in the base into small piecewise smooth regions $R_i$ so that the corresponding lifts are entirely contained in a region where the plane field is either tangent to $\mathcal{F}$ or is smooth. One then factorises $\gamma$ as a product of loops of the form 
$$\gamma = \prod_{i=1}^N \tau^{-1}_i\gamma_i\tau_i = \prod_{i=1}^N\beta_i,$$ where $\tau_i$ is an arc and $\gamma_i$ parametrises $\partial R_i$ in the positive sense. The lifts of each $\gamma_i$ that lies in the integrable region have trivial holonomy and lifts to a region where $\xi$ is a smooth confoliation have non-increasing (i.e.\ non-positive) holonomy (cf.\ \cite{ETh}, pp.\ 13--14). In particular, the holonomy around $\gamma_i$ is non-positive in total and the same is then true of $\beta_i$ since conjugation by a path only changes the holonomy by a conjugation. Since the holonomy around $\gamma$ is just the product of the holonomies around the $\beta_i$, this holonomy is also non-positive. Moreover, if the lift $\tilde{\gamma}_x$ of $\gamma$ through $x$ passes through the contact region, then the holonomy $h$ satisfies $h(x)<x$.
\begin{rem}\label{hol_neg}
In order to assume that holonomies are globally well defined on $D^2 \times [0,1]$ one can assume that the confoliation is integrable near $D^2 \times\{0,1\}$.  We also remark that the holonomies of the characteristic foliation on the boundary of any piecewise smooth polyhedron whose boundary is transverse to a \emph{contact structure} will also necessarily be negative near `supporting vertices' (cf.\ Section \ref{jiggling_sec}).
\end{rem}
There is another generalisation of a smooth confoliation that will be useful below. 

\begin{defn}\label{tangential_conf}
Let $\xi$ be tangent to a \textbf{smooth} vector field $Y$ on some subset $U = J \times [0,1] \times [0,1] \subset M$, where $J = [0,1]$ or $S^1$ so that $Y$ is identified with the coordinate vector field $\frac{\partial}{\partial y}$ given by the second coordinate. Then we will call $\xi$ a \textbf{tangential $C^0$-confoliation} with respect to $Y$, if in terms of the coordinates $(x,y,z)$ the plane field $\xi$ is the kernel of
$$dz-f(x,y,z)dx$$
such that $f$ is (weakly) monotone in $y$ and the kernel of the $1$-form $dz-f(x,0,z)dx$ is tangent to a foliation. 
\end{defn}
Note that if $f(x,y,z)$ is strictly monotone on $\mathcal{O}p(K)$ for some compact subset $K$, then we can smoothen $\xi$ on $\mathcal{O}p(K)$ in such a way that the function $f(x,y,z)$ remains strictly monotone in $y$ on $K$. Observe further that in the case $J = [0,1]$ the holonomy around $\partial ([0,1]\times [0,1])$ is non-negative and is strictly negative if $f(x,y,z)$ is strictly monotone on some part of each interval $\{x \}\times [0,1] \times \{z\}$. 

\section{Smoothing $C^0$-Foliations near leafwise arcs and near the $1$-skeleton}\label{sec_smooth}
At various points during the process of approximating a general $C^0$-foliation we will need to ensure that the foliations we are considering are smooth on certain regions of the given manifold. In general, one cannot approximate a $C^0$-foliation by a smooth one globally. However applying standard smoothing theory we can obtain smoothings of $C^0$-foliations on neighbourhoods of leafwise arcs and near the $1$-skeleton of an appropriately transverse triangulation.
\begin{lem}[Smoothing near leafwise arcs]\label{smooth}
Let $\mathcal{F}$ be a $C^0$-foliation on a $3$-manifold $M$ and let $\gamma$ be a smoothly embedded compact arc in a leaf. Then there is a $C^0$-foliation $\mathcal{F}'$ such that:
\begin{enumerate}
\item $\mathcal{F}'$ agrees with the original foliation outside a neighbourhood of $\gamma$;
\item $\mathcal{F}'$ is smooth on a neighbourhood of $\gamma$;
\item $\mathcal{F}'$ is $C^0$-close to $\mathcal{F}$ (i.e.\ their tangent distributions are close).
\end{enumerate}
\end{lem}
\begin{proof}
Let $N(\tilde{\gamma} ) = \tilde{\gamma} \times [-1,1] \times [\epsilon,\epsilon]$ be a smooth neighbourhood of a slight extension $\tilde{\gamma}$ of the arc $\gamma$ so that $\tilde{\gamma} \times [-1,1] \times \{0\}$ lies in a leaf $L_0$ and the interval fibres corresponding to the third factor are transverse to $\mathcal{F}$. Then the foliation is given as a family of graphs of functions
$$(x,y,z) \longmapsto (x,y,f_z(x,y))$$
so that the partial derivatives of $f_z(x,y)$ in the $x$ and $y$ directions are continuous. Write $F(\overline{x}) = F(x,y,z)= f_z(x,y)$. We then smoothen $f_z(x,y)$ using convolution with a bump function $\rho$ such that $\int_{\R^3} \rho =1$: Set $\rho_{\delta}(\overline{x}) = \delta^{-3}\rho\left(\overline{x}/\delta\right)$ and define
$$G_{\delta}(\overline{x}) = \int_{\R^3}\rho_{\delta}\left(\overline{x} -\overline{y}\right)F(\overline{y})d\overline{y}= \int_{\R^3}\rho_{\delta}\left(\overline{z}\right)F(\overline{x}-\overline{z})d\overline{z}.$$
Then the tangent planes of the graphs of these approximations approach the tangent planes of $\mathcal{F}$. This is the case since the partial derivatives of $G_{\delta}$ can be computed by differentiating under the integral sign and since the partial derivatives of $F$ in the $x$ and $y$ directions are continuous. Note also that convolution preserves the inequality $F(x,y,z_1) < F(x,y,z_2)$ for $z_1<z_2$, so that the graphs of the $G_\delta$ do indeed define a foliation, in the sense that the map
$$(x,y,z) \longmapsto (x,y,G_\delta(x,y,z))$$
is a smooth bijection. Note however that $\frac{\partial G_\delta}{\partial z}$ is only non-negative and may not be strictly positive even though $G_\delta(x,y,z)$ is strictly monotone for fixed $(x,y)$. We remedy this by setting
$$\widehat{G}_{\delta} = G_\delta(x,y,z) + \delta \cdot z,$$
which is then gives a diffeomorphism as $\frac{\partial \widehat{G}_{\delta}}{\partial z} > 0$. Let $\sigma$ be a \emph{fixed} bump function with support in $N(\tilde{\gamma})$ that is identically one on a neighbourhood of $\gamma$ and set
$$h_{\delta} = (1-\sigma)\cdot F + \sigma\cdot \widehat{G}_{\delta} .$$
We then modify the foliation on $N(\tilde{\gamma})$ by taking the images of the $(x,y)$-planes under the map
$$(x,y,z) \longmapsto (x,y,h_{\delta}(x,y,z))$$
on $N(\tilde{\gamma})$ which glues together with the original foliation. This then gives the desired approximation for $\delta$ sufficiently small. 
\end{proof}

\begin{rem}
Note that if $N = I \times I \times I $ is any regular neighbourhood of $\gamma $ so that the intervals of the last factor are transverse to $\mathcal{F}$, then we may smoothen on a slightly smaller neighbourhood $N' \subseteq N$. In particular, we can smoothen near any given transversal arc or circle.
\end{rem}

\begin{rem}\label{smoothing_attempt}
The statement of Lemma \ref{smooth} seems to give a recipe for approximating any $C^0$-foliation by one that is smooth (which is in general impossible): one simply tries to smoothen with respect to a finite foliated atlas one chart at a time. However, the representation of the foliation as a family of graphs of functions, requires a choice of the parameter $z$ or equivalently the choice of a local transversal. When one changes coordinates, one must also change the parameter $z$ and for a $C^0$-foliation this transformation will in general only be continuous. Hence when one tries to continue the smoothing process one cannot do this in a way that is compatible with smoothings on other charts. What is achieved through such local patching of smoothings is that the leaves of resulting foliation are \emph{leafwise} $C^{\infty}$-immersed and vary continuously in the $C^{\infty}$-topology. Hence we can always assume that this is the case after a $C^0$-approximation.
\end{rem}

We now extend Lemma \ref{smooth} to smoothen integrable plane fields near the $1$-skeleton of a suitable triangulation.
\begin{lem}[Smoothing near the $1$-skeleton]\label{smooth_1_skel}
Let $ \mathcal{F}$ be a $C^0$-foliation and suppose that $\mathcal{F}$ is transverse to the $1$-skeleton of some triangulation. Then $\mathcal{F}$ can be $C^0$-approximated by a foliation $\mathcal{F}'$ that is smooth on some {\bf fixed} neighbourhood $U$ of the $1$-skeleton, where $U$ does not depend on how good the approximation is. 
\end{lem}
\begin{proof}
Near vertices we can apply the convolution argument of Lemma \ref{smooth} \emph{verbatim}. We then choose tubular neighbourhoods $$\nu_{-\epsilon}(e) \cong D^2 \times [0,1]$$ of each edge $e$ with $\epsilon$ neighbourhoods of the ends points removed, so that $0 \times [0,1] \subseteq e$, $\nu_{-\epsilon}(e)$ intersects the neighbourhoods of the vertices where $\xi$ has been smoothened and the neighbourhoods are pairwise disjoint for distinct edges. We can assume that the intervals $pt \times[0,1]$ are transverse to $\mathcal{F}$ and hence to $\mathcal{F}'$ by assuming that $\mathcal{F}'$ is $C^0$-close to $\mathcal{F}$. We can also assume that the discs $D^2 \times pt$ are tangent to $\mathcal{F}'$ near the end points of $[0,1]$, as $\mathcal{F}'$ has been made smooth there already. We then apply the convolution argument again and note that convolution preserves the function $f_z = z+C$ so that the resulting foliation agrees with the given one near vertices. Moreover, instead of perturbing the function $G_{\delta}$ as in the proof of Lemma \ref{smooth} by adding $\delta z$ we can add $\delta\rho(z)z$ instead for a suitable cut-off function with support in the interior of $[0,1]$, so that the resulting foliation agrees with the given one near the endpoints of $e$. In this way we obtain the desired approximation of $\mathcal{F}$ which is smooth on the chosen neighbourhood of the $1$-skeleton.
\end{proof}
\section{Nice annular fences}\label{nice_section} 
It will be convenient to choose special ``nice'' coordinates near simple closed curves lying on a leaf of a foliation $\mathcal{F}$. The purpose of this section is to describe precisely what this niceness will be and to show that such coordinates always exist, at least for a $C^0$-dense set of foliations.

An embedded annulus $A = S^1 \times [-\epsilon,\epsilon]$ whose core $\gamma = S^1 \times \{0\}$ lies on a leaf of a foliation $\mathcal{F}$ such that the interval fibers are transverse to $\mathcal{F}$ will be called an {\bf annular fence}. We can approximate as in Lemma \ref{smooth} to assume that $\mathcal{F}$ is smooth near a fixed fiber of $A$ say $0 \times [-\epsilon,\epsilon]$. We then consider a product neighbourhood given by flowing along a smooth vector field $Y$ that is $C^0$-close to $\mathcal{F}$ and is tangent both to $\mathcal{F}$ near $0 \times [-\epsilon,\epsilon]$ and to the leaf containing $\gamma$. We denote the resulting coordinates $(x,y,z)$, where $\frac{\partial}{\partial y}$ is tangent to the flow of $Y$. 

Next let $h$ be the holonomy of the foliation $\mathcal{F}$ along $\gamma$ with respect to a closed transversal $0 \times [-\epsilon',\epsilon']$ where $\epsilon' < \epsilon$ is such that $h( [-\epsilon',\epsilon']) \subset [-\epsilon,\epsilon]$. Now consider a ``nice'' suspension foliation $\mathcal{F}_{nice}$ on $A$ given by $h$. By \textbf{nice} we mean that for a parametrisation $t \mapsto (t, \gamma_z(t))$ of the leaf $L_z$ through $ z \in [-\epsilon,\epsilon']$ we have (cf.\ Figure \ref{nice}):
\begin{itemize}
\item $\gamma'_z(t) =0 $ if $\gamma_z(t) \in \mathcal{O}p(0 \times [-\epsilon,\epsilon])$;
\item $\gamma'_z(t)\ge 0$ and $\gamma'_z(t)> 0$ away from $0 \times [-\epsilon,\epsilon]$, if $z < h(z)$;
\item $\gamma'_z(t) \le 0$  and $\gamma'_z(t) < 0$ away from $0 \times [-\epsilon,\epsilon]$, if $z > h(z)$;
\item $\gamma'_z(t) \equiv 0 $, if $z = h(z)$.
\end{itemize}
\begin{figure}[h]\begin{center}
\input{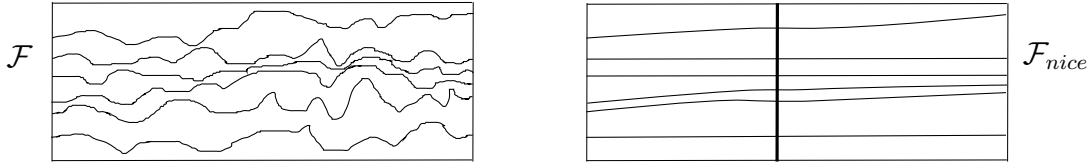} \caption{The foliation on the annulus (cut open along an interval) on the left is not nice at all, whereas that on the right is very nice. The thickened interval represents $0 \times [-\epsilon,\epsilon]$, near which $\mathcal{F}_{nice}$ is horizontal.}\label{nice}
\end{center}
\end{figure}
In particular, $\gamma_0$ is a parametrisation of the holonomy curve $\gamma$ which then corresponds to $S^1 \times \{0\}$. Note that both $\mathcal{F}$ and $\mathcal{F}_{nice}$ are by construction tangent to $z = 0$ near $0 \times [-\epsilon,\epsilon]$. We can then straighten out the foliation using straight line isotopies along the $z$-fibers relative to some neighbourhood $N$ of $0 \times [-\epsilon,\epsilon]$ so that the resulting foliation is a product of $\mathcal{F}_{nice}$ with the $y$-intervals and agrees with $\mathcal{F}$ away from the annulus $A$. More precisely, away from some small neighbourhood $N' \subseteq N= \mathcal{O}p(0 \times [-\epsilon,\epsilon]) \subseteq A$ the foliation $\mathcal{F}$ is given by functions
$$f_z(x,y) \co D^2 \longrightarrow \R$$
which equal $z$ on the overlap with $N$ and have $f_0 = 0$. On the same subset the product of $\mathcal{F}_{nice}$ with the $y$-intervals is given by functions 
$$f^{nice}_z(x,y) \co D^2 \longrightarrow \R$$
which also equal $z$ on the overlap with $N$ and again satisfy $f^{nice}_0 = 0$. One then interpolates between $f_z(x,y)$ and $f^{nice}_z(x,y)$ using a cut-off function as in the proof Lemma \ref{smooth}. For this it is convenient to take a cut-off function of the form
$$\sigma_\eta(x,y,z) = \rho(x,y)\cdot \gamma_{\eta}(z)$$
where $\gamma_{\eta}$ has support on $(-\eta,\eta)$. Note that the $x,y$ derivatives of $\sigma_\eta$ are bounded independently of $\eta$. By taking $\eta$ sufficiently small the resulting foliation $\mathcal{F}_\eta$ can be made arbitrarily close to $\mathcal{F}$. We will call coordinates $(x,y,z)$ near an annular fence {\bf very nice} if the induced characteristic foliation $\mathcal{F}(A)$ on $A = \{(x,y,z) \ | \ y = 0\}$ given by intersecting leaves of $\mathcal{F}$ with $A$ is nice and the foliation is tangent to the $y$-intervals near $A$. We have seen above that we can always find very nice coordinates after a small perturbation and we note this in the following:
\begin{lem}[Nice coordinates]\label{nice_approx}
Let $\gamma$ be an embedded curve in a leaf of a $C^0$-foliation $\mathcal{F}$. Then $\mathcal{F}$ can be $C^0$-approximated by foliations that are very nice for suitable coordinates on some neighbourhood of $\gamma$.
\end{lem} 
Note that the size of the neighbourhood given in Lemma \ref{nice_approx} is not fixed -- its height will be very small if the resulting foliation is to be close to $\mathcal{F}$. We will call a fence {\bf nice at the boundary}, if its boundary is either transverse or equal to the $z$-levels along the boundary. Note also that a very nice annulus can always be made nice near the boundary by a $C^{0}$-small perturbation.
\begin{rem}\label{rem_nice_fences}
It is easy to see that blowing up leaves and inserting holonomy can be assumed to preserve niceness at the boundary. This is obvious in the case when the boundary is transverse, since this is $C^0$-stable. If a boundary curve $\gamma$ is a $z$-level then we can blow up in such a way that the blow up near $\gamma$ just looks like a product $\gamma \times [0,\epsilon]$. We then insert holonomy so that the foliation remains nice. Furthermore, if a fence $A$ is nice near the boundary and the characteristic foliation is transverse to $\partial A$, then we can 
$C^0$-approximate $\mathcal{F}$ by foliations that are smooth near $\partial A$ as in Section \ref{sec_smooth}.
\end{rem}
\section{Producing contact regions near curves with contracting holonomy}\label{contact_regions}
Given an embedded curve $\gamma$ lying in a leaf of a foliation $\mathcal{F}$ one defines the holonomy of the foliation around $\gamma$, by pushing a small transversal starting at a base point along the leaves of $\mathcal{F}$ and considering the (germ) of the first return map. We will say that a curve has non-trivial holonomy if this return map $h$ is non-trivial germinally.

Suppose that the holonomy $h$ around some curve is contracting for some interval, i.e. $h(I)$ is properly contained in the interior of $I$ for some closed interval $I$ transverse to $\mathcal{F}$. Note that this is equivalent to the existence of a normal annular fence $A$ so that the induced oriented characteristic foliation points into $A$ along its boundary. We shall call such an annulus an {\bf annular fence with contracting holonomy}. If the opposite inclusion holds, i.e. $h(I)$ contains $I$ in its interior, then we will say that $h$ is repelling for some interval. A special case of this is when the holonomy is \textbf{sometimes attracting}, which means that there are sequences $t_n^+,t_n^-$ of positive resp.\ negative numbers so that
$$t_n^+ \searrow 0 \textrm{ and } h(t_n^+) < t_n^+  \quad , \quad t_n^- \nearrow 0 \textrm{ and } h(t_n^-) > t_n^-.$$
The holonomy is called sometimes repelling if the inequalities are reversed. Note that the condition that $h$ is sometimes attracting is equivalent to the existence of arbitrarily small contracting intervals. We will also say that the holonomy is sometimes attracting on one side if only one of these inequalities holds.

Given a contracting annular fence we next explain how to produce a contact structure close to the fence. In order to do this we will use a weakened version of the corresponding result in the smooth case (cf.\ \cite{ETh} Proposition 2.5.1; see also \cite{Pet}).

\begin{lem}\label{cont_holonomy}
Let $\mathcal{F}$ be a $C^0$-foliation on $M$ and let $A$ be a contracting annular fence. Assume that $\mathcal{F}$ is very nice with respect to a choice of smooth coordinates $(x,y,z)$ on a neighbourhood of $N(A) = S^1 \times [0,1] \times [-\eta,\eta]$ and that $\mathcal{F}$ is smooth on a neighbourhood $U$ of the horizontal boundary $S^1  \times [0,1] \times\{\pm\eta\}$. 

Then for $A' \subset A$ a slightly smaller fence, there is $C^0$-family of $C^0$-confoliations $(\xi_t)_{t \in [0,\delta]}$ on $ M \setminus N(A')$ so that 
\begin{itemize}
\item $\xi_t$ is tangent to $\mathcal{F}$ away from $\mathcal{O}p(N(A))$;
\item $\xi_0$ is everywhere tangent to $\mathcal{F}$;
\item $\xi_t$ is smooth and contact on some neighbourhood $\mathcal{O}p(\partial A)$ for $t>0$. 
\end{itemize}
Moreover, there is a smooth family of vector fields $(Y_t)_{t\in [0,\epsilon]}$ on $\mathcal{O}p(N(A))$ that are tangent to $\xi_t$ for each $t$ such that $Y_0 = \frac{\partial}{\partial y}$ and $Y_t =  \frac{\partial}{\partial y}$ outside a slightly smaller neighbourhood $\mathcal{O}p'(N(A)) \subset \mathcal{O}p(N(A))$.
\end{lem}
\begin{proof}
With respect to the very nice coordinates $(x,y,z)$ on a neighbourhood of $N(A)$ the foliation $\mathcal{F}$ is tangent to the plane field given by the kernel of the $1$-form:
$$\alpha_{nice} = dz - f(x,z)dx.$$
The assumption that the fence is contracting means that, after a small perturbation to make things nice at the boundary, we have $f(x,\eta) < 0$ and $f(x,-\eta)>0$. Now the foliation given by $\alpha_{nice}$ is \emph{diffeomorphic} to any other foliation pointing into $A$ on a neighbourhood of $\partial A$. In particular, we can assume (after applying a suitable diffeomorphism) that $\frac{\partial f}{\partial z}(x,\eta) >0$ and $\frac{\partial f}{\partial z}(x,-\eta) <0$ -- in other words $f$ is increasing in $z$ near the upper boundary component and decreasing near the lower boundary component.

We let $\hat{A}$ be a slight extension of $A$ and let 
$$N(\hat{A}) = S^1 \times [-\epsilon ,1 + \epsilon] \times [-\eta - \epsilon,\eta+\epsilon]$$
be a neighbourhood containing $N(A)$. Now define a smooth family of graph-like diffeomorphisms $\Phi_t(x,y,z) = (x,y, g_t(y,z))$ on $V_{\epsilon} = S^1 \times [-\epsilon,0] \times [-\eta-\epsilon,\eta+\epsilon]$ with the following properties:
\begin{itemize}
\item $\Phi_0 = id$;
\item $\Phi_t = id$ if $-\epsilon \le y \le -\epsilon/2$ or $|z\pm(\eta +\epsilon)| < \epsilon/2$;
\item $g_t(y,z)\le z$ and for $t>0$ the inequality is strict if $-\epsilon/3 \le y\le 0$ and $|z\pm (\eta+\epsilon)| > \epsilon/2$;
\item The map $g_t(y,z)$ is independent of $y$ for all $-\epsilon/3 \le y\le 0$;
\item For all $t>0$ we have $\frac{\partial g_t}{\partial z}<1$, if $-\epsilon/2 < z-\eta < \epsilon/2$ and $\frac{\partial g_t}{\partial z}>1$, if $-\epsilon/2< z + \eta < \epsilon/2$.
\end{itemize}
The above conditions codify the fact that the map $g_t$ pushes everything down and has a particular form near $z = \pm \eta$. Then on $S^1 \times [-\epsilon/3,0] \times [-\eta-\epsilon,\eta+\epsilon]$ the pullback satisfies:
$$\frac{1}{g'_t}\Phi^*_t\alpha_{nice} = \frac{1}{g'_t}\left(g'_tdz - f(x,g_t(z))dx\right) = dz -\frac{f(x,g_t(z))}{g'_t}dx = dz - \hat{f}_t(x,z)dx,$$
where $g'_t$ denotes $\frac{\partial g_t}{\partial z}$. In particular we have $\hat{f}_t(x,z)\le f(x,z)$ and the inequality is strict if $ |z\pm\eta| < \epsilon/2$ due to the properties described above. Thus we can interpolate between $\hat{f}_t(x,z)$ and $f(x,z)$ by a function $F_t(x,y,z)$ that is strictly monotone in $y$ for $-\epsilon/4 < y<1+\epsilon/4 $ and $z$ close to $\pm \eta$. The kernels of the family of $1$-forms
$dz - F_t(x,y,z)dx$
gives the desired $C^0$-confoliations on $N(A) \setminus N(A') $, where we set
$$N(A') = S^1 \times [\epsilon,1-\epsilon]\times [-\eta+\epsilon/3,\eta-\epsilon/3].$$
We finally extend this confoliation by $T\mathcal{F}$ outside of $N(\hat{A)}$ and set $Y_t = \Phi_t^* \frac{\partial}{\partial y}$, which is obviously tangent to $\xi_t$.
\end{proof}
\begin{rem}\label{repelling_case}
The proof of Lemma \ref{cont_holonomy} holds equally well for a repelling fence: one simply swaps the orientation of $\mathcal{F}$ and uses the same argument.
\end{rem}
\begin{rem}\label{y_coord}
Although the contact structures $\xi_t$ in Lemma \ref{cont_holonomy} are no longer tangent to the given $y$-coordinate vector field, this can in fact be assumed to be the case after a $C^{\infty}$-small coordinate change near $N(A)$ given by flowing along $Y_t$.
\end{rem}


Once we have produced enough contact regions we will need to transport this non-integrability along curves that end in a contact region. This is described precisely in the following lemma:
\begin{lem}[\cite{ETh} Lemma 2.8.2]\label{transport}
Let $V = I \times I \times I \subseteq \R^3$, where $I = [0,1]$ and suppose that we have an orientation preserving embedding into $M$ so that each $y$-interval $x \times I \times z$ is tangent to a positive smooth confoliation $\xi$ which is contact near the positive end $I \times \{1\} \times I$. 
 Then there is a smooth family of confoliations $(\hat{\xi}_t)_{t \in [0,1]}$ with $\hat{\xi}_0 = \xi$ so that $\xi_t$ is contact on a slightly smaller set $V' \subseteq V$ for any $t >0$ and agrees with $\xi$ outside of $V$. 
 \end{lem}
 
\noindent Moreover, the same statement holds parametrically, since the condition of being (strictly) monotone is convex. Note that the coordinates in this case also depend on the parameter.
\begin{lem}[Parametric Version]\label{transport_para}
Let $(\xi_\sigma)_{\sigma \in [0,\delta]}$ be a smooth family of confoliations on $V \cong I \times I \times I \subseteq M$ and suppose that each $y_\sigma$-interval $x_\sigma \times I \times z_\sigma$ is tangent to $\xi_\sigma$ and that $\xi_\sigma$ is contact near the positive end $I \times \{1\} \times I$. Then there is a smooth family of confoliations $\hat{\xi}_{(\sigma,t)}$ with $\hat{\xi}_{(\sigma,0)} = \xi_\sigma$ so that $\xi_{(\sigma,t)}$ is contact on a slightly smaller subset $V' \subseteq V$ for any $t >0$ and agrees with $\xi_\sigma$ outside of $V$.
\end{lem}

\subsection*{Extending contact structures on thickened annuli}
In the final step of our approximation process we will need to fill in a contact structure on a thickened annulus $S^1 \times [0,1] \times [-\eta,\eta]$ under certain assumptions on the way the contact structure meets the boundary.
\begin{lem}[Filling in holes]\label{movie_quant}
Let $A \times [-\eta,\eta]= S^1 \times [-1,1]\times [-\eta,\eta]$ be a thickened annulus and let $(x,y,z)$ be smooth product coordinates. Assume that $\xi$ is a contact structure defined near $\partial (A \times [-\eta,\eta])$ such that:
\begin{itemize}
\item $\xi$  is tangent to the $y$-fibers and is $\epsilon$-$C^0$-close to the product foliation given by the $\{z = C\}$-slices;
\item The characteristic foliations on each annular slice $A_{y} = A \times \{y\}$ induced by $\xi$ are inward pointing along the boundary.
\end{itemize}
Then there is a contact structure $\widehat{\xi}$ extending $\xi$ to all of $A \times [-\eta,\eta]$ that is $(8\eta+ \epsilon)$-close to the product foliation in the $C^0$-sense.
\end{lem}
\begin{proof}
Note that in terms of the coordinates $(x,y,z)$ the contact structure $\xi$ is given as the kernel of a $1$-form:
$$dz - a(x,y,z)dx.$$
The assumptions on $\xi$ mean that $a(x,y,\eta) <0, a(x,y,-\eta) >0$ and $|a(x,y,z)| < \epsilon$. It is then easy to define a function $b^-(x,z)$ on the annulus $A_{-1/2}$ that has the following properties (see Figure \ref{Graph_caption}):
\begin{itemize}
\item $|b^-(x,z)| < \epsilon$;
\item $a(x,-1,z) \le b^-(x,z)$;
\item  $b^-(x,z) = a(x,-1/2,z)$ for $z$ close to $\pm \eta$;
\item For each fixed $x$ we have $b^-(x,z) =0$ precisely on $I^- = [\eta - \delta,\eta-\delta/2]$ for some (sufficiently) small $\delta$. (Note: this interval does \emph{not} depend on $x$)
\end{itemize}
We then interpolate between $a(x,-1,z)$ and $b^-(x,z)$ to obtain a function $\hat{a}(x,y,z)$ on $S^1 \times [-1,-1/2]\times [-\eta,\eta]$ that is monotone in $y$ so that:
\begin{itemize}
\item The kernel of the $1$-form $dz - \hat{a}(x,y,z)dx$ is a confoliation;
\item $\hat{a}(x,y,z) = a(x,y,z)$ near the horizontal boundary  $z = \pm \eta$ as well as near the annulus $S^1 \times \{-1\} \times [-\eta,\eta]$;
\item The kernel of $dz - \hat{a}(x,y,z)dx$ is $\epsilon$-close to the product foliation;
\item $\hat{a}(x,y,z)$ is constant in the $y$-direction for $y \in [-3/4,-1/2]$ away from an arbitrarily small neighbourhood of the horizontal boundary.
\end{itemize}
We can also extend the contact structure to a confoliation in a similar way on $S^1 \times [1/2,1]\times [-\eta,\eta]$ so that the characteristic foliation on $A_{1/2}$ has precisely one interval of closed orbits as well. We denote the corresponding function on $A_{1/2}$ by $b^+(x,z)$.

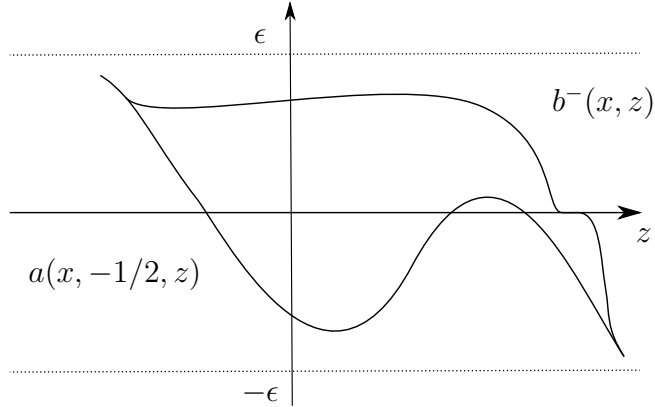
\begin{figure}[h]\label{Graph}
\psset{xunit=.5pt,yunit=.5pt,runit=.5pt}
\begin{pspicture}(477.59997559,305.78060913)
\rput(450,230){$b^-(x,z)$}  \rput(80,100){$a(x,-1/2,z)$} \rput(190,280){$\epsilon$} \rput(190,10){$-\epsilon$} \rput(480,130){$z$}
{
\newrgbcolor{curcolor}{0 0 0}
\pscustom[linewidth=0.85399997,linecolor=curcolor]
{
\newpath
\moveto(212.59798,305.35360913)
\lineto(214.30256,0.42700913)
}
}
{
\newrgbcolor{curcolor}{0 0 0}
\pscustom[linestyle=none,fillstyle=solid,fillcolor=curcolor]
{
\newpath
\moveto(212.64571898,296.81374284)
\lineto(216.08076109,293.41689192)
\lineto(212.59798,305.35360913)
\lineto(209.24886806,293.37870073)
\lineto(212.64571898,296.81374284)
\closepath
}
}
{
\newrgbcolor{curcolor}{0 0 0}
\pscustom[linewidth=0.85399997,linecolor=curcolor]
{
\newpath
\moveto(212.64571898,296.81374284)
\lineto(216.08076109,293.41689192)
\lineto(212.59798,305.35360913)
\lineto(209.24886806,293.37870073)
\lineto(212.64571898,296.81374284)
\closepath
}
}
{
\newrgbcolor{curcolor}{0 0 0}
\pscustom[linewidth=1.09232843,linecolor=curcolor]
{
\newpath
\moveto(88.19442,233.90680913)
\curveto(110.90915,208.10706913)(292.38976,254.91169913)(357.28358,227.16652913)
\curveto(414.46385,202.63043913)(404.99539,146.87405913)(418.41912,146.69870913)
\curveto(431.84285,146.52336913)(413.52199,146.69870913)(430.47208,146.69870913)
\curveto(447.42217,146.69870913)(447.60807,111.72536913)(451.22201,89.76852913)
\curveto(454.83595,67.81168913)(451.34204,60.04310913)(465.30163,37.68576913)
}
}
{
\newrgbcolor{curcolor}{0 0 0}
\pscustom[linewidth=1.09232843,linecolor=curcolor]
{
\newpath
\moveto(465.30163,37.68576913)
\curveto(425.67713,99.40554913)(372.19977,231.73287913)(304.70131,106.03187913)
\curveto(237.20285,-19.66911087)(161.1722,132.99960913)(141.32189,158.15774913)
\curveto(121.47158,183.31588913)(101.37803,219.35379913)(88.19442,233.90680913)
\curveto(75.01081,248.45981913)(69.01501,250.40214913)(69.01501,250.40214913)
}
}
{
\newrgbcolor{curcolor}{0 0 0}
\pscustom[linewidth=1.09200001,linecolor=curcolor]
{
\newpath
\moveto(477.05399,146.69870913)
\lineto(0.546,146.69870913)
}
}
{
\newrgbcolor{curcolor}{0 0 0}
\pscustom[linestyle=none,fillstyle=solid,fillcolor=curcolor]
{
\newpath
\moveto(466.13398992,146.69870913)
\lineto(461.76598989,142.3307091)
\lineto(477.05399,146.69870913)
\lineto(461.76598989,151.06670916)
\lineto(466.13398992,146.69870913)
\closepath
}
}
{
\newrgbcolor{curcolor}{0 0 0}
\pscustom[linewidth=1.09200001,linecolor=curcolor]
{
\newpath
\moveto(466.13398992,146.69870913)
\lineto(461.76598989,142.3307091)
\lineto(477.05399,146.69870913)
\lineto(461.76598989,151.06670916)
\lineto(466.13398992,146.69870913)
\closepath
}
}
{
\newrgbcolor{curcolor}{0 0 0}
\pscustom[linewidth=0.97799999,linecolor=curcolor,linestyle=dashed,dash=0.978 1.956]
{
\newpath
\moveto(477.05399,267.41300913)
\lineto(0.546,265.98440913)
}
}
{
\newrgbcolor{curcolor}{0 0 0}
\pscustom[linewidth=0.97799999,linecolor=curcolor,linestyle=dashed,dash=0.978 1.956]
{
\newpath
\moveto(477.05399,27.41300913)
\lineto(0.546,25.98440913)
}
}
\end{pspicture} \caption{The functions $a(x,-1/2,z)$ and $b^-(x,z)$ for a given $x$.}\label{Graph_caption}
\end{figure}

Let $z_0 \in I^-$ and consider the interval $J_{\tau,\sigma} =[-\eta+\tau(\eta + z_0),\eta-\sigma(\eta - z_0)]$ for any $\tau,\sigma \in (0,1]$. We next choose a smooth family of diffeomorphisms $f_{\tau,\sigma}(z)$ on $[-\eta,\eta]$ parametrised by $\tau,\sigma \in (0,1]$ with the following properties:
\begin{itemize}
\item $f_{\tau,\sigma}(z) = id$ near $\pm \eta$ and $f_{1,1} = id$;
\item $f_{\tau,\sigma}(z) (J_{\tau,\sigma}) \subseteq I^-$;
\item $f'_{\tau,\sigma}(z) \ge 1$ if $f_{\tau,\sigma}(z) \ne I^-$. 
\end{itemize}
Then pulling back the form $dz -b^-(x,z)dx$ under $f_{\tau,\sigma}$ and dividing by $f'_{\tau,\sigma}(z)$ gives
$$dz-\frac{1}{f'_{\tau,\sigma}(z)}b^-(x,f_{\tau,\sigma}(z))dx = dz-\hat{b}^-_{\tau,\sigma}(x,z)dx$$
on the annulus $A$, which we identify with $A_{-1/2}$. By construction we have $|\hat{b}^-_{\tau,\sigma}(x,z)| < \epsilon$ for any $\tau,\sigma \in (0,1]$. Note that as $\tau,\sigma$ increase, pulling back under $f_{\tau,\sigma}$ has the effect of increasing the length of the interval of zeroes and making the function very small away from an increasingly small neighbourhood of the boundary. In a similar way we define a family of diffeomorphisms $g_{\tau,\sigma}(z)$ so that 
$$|\hat{b}^+_{\tau,\sigma}(x,z)| = \left|\frac{1}{g'_{\tau,\sigma}(z)}b^+(x,g_{\tau,\sigma}(z))\right| < \epsilon.$$
By choosing $(\tau_0,\sigma_0)$ and $(\tau_1,\sigma_1)$ with $\tau_0 < \tau_1$ resp.\ $\sigma_0 > \sigma_1$ appropriately we can assume that $\hat{b}^-_{\tau_0}(x,z) \le \hat{b}^+_{\sigma_0}(x,z)$ and that the inequality is strict on a neighbourhood of the boundary which intersects the region on which $\hat{a}(x,y,z)$ is constant in $y$. We can clearly assume that for fixed $z$ the function $f_{\tau,\sigma}(z)$ is affine in $\tau$ and $\sigma$ (individually) on $[\tau_0,1] \times [\sigma_0,1]$.

We next define a graph-like diffeomorphism $\Phi^-$ on $S^1 \times [-1,-1/2]\times [-\eta,\eta]$  by
$$(x,y,z) \longmapsto (x,y , f_{(\varphi(y),\psi(y))}(z))$$
for an appropriate choice of non-increasing functions 
$$\varphi\co [-3/4,-1/2]\longrightarrow [\tau_0,1] \quad ,\quad \psi \co [-3/4,-1/2]\longrightarrow [\sigma_0,1]$$
which are constant near the end points and satisfy 
$$|\varphi'(y)| \le 2|1- \tau_0| < 2 \textrm{ resp.\ } |\psi'(y)| \le 2|1- \sigma_0| < 2.$$  

 \noindent Note that since the partial derivatives of $f =f_{\tau,\sigma}(z)$ with respect to $\tau,\sigma$ are constant and bounded by $|f_{\tau_0,\sigma_0}(z) - z| < 2\eta$. We thus have estimates
\begin{align*}
\left|\frac{\partial \Phi^-}{\partial y \ }\right| &\le |\varphi'(y)|\left|\frac{\partial f}{\partial \tau}(\varphi(y),\psi(y))(z)\right| + |\psi'(y)|\left|\frac{\partial f}{\partial \sigma}(\varphi(y),\psi(y))(z)\right| < 8\eta\\
\end{align*}
and the new plane field given via pullback is spanned by
$$X = \partial_x + \hat{b}^-_{\tau_0}(x,z)\partial_z\textrm{ and }Y = \partial_y + \frac{\partial \Phi^-}{\partial y \ }\partial_z$$
at each point at which the map $\Phi^-$ is non-trivial. It follows that the planes spanned by these vectors are $(8\eta +\epsilon)$-close to the product foliation, since the cross product $X \times Y$ lies at a distance at most $8\eta +\epsilon$ from $\partial_z$.  Similarly for appropriate step functions $\alpha,\beta$ we can define $\Phi^+$ on $S^1 \times [\frac{1}{2},1]\times [-\eta,\eta]$ by
$$(x,y,z) \longmapsto (x,y,g_{ (\alpha(y),\beta(y))}(z)).$$
In the same way the plane field given by pulling back under $\Phi^+$ is $(8\eta +\epsilon)$-close to the product foliation. We can then interpolate between $ \hat{b}^-_{\tau_0}(x,z) $ and $ \hat{b}^+_{\sigma_0}(x,z)$ to obtain a function that is (weakly) monotone in $y$ so that the kernel of the resulting $1$-form
$$dz - \bar{a}(x,y,z)dx$$
is a confoliation that agrees with $\xi$ near the boundary and is $(8\eta +\epsilon)$-close to the product foliation. We finally perturb this confoliation to a contact structure using Lemma \ref{transport} (or perhaps just common sense).
\end{proof}
\begin{rem}
Note that we have used the standard Euclidean norm in Lemma \ref{movie_quant}. In general the bound is given by $C \cdot (8\eta +\epsilon)$  for some constant $C$ depending only on the initial choice of (very nice) coordinates.
\end{rem}

\section{Polyhedral decompositions and Jiggling}\label{jiggling_sec}
Following ideas of Colin \cite{Col} we now describe how to modify plane fields near the $2$-skeleton of a polyhedral decomposition so that they under certain holonomy assumptions that can be extended to (tight) contact structures over the $3$-cells.
\subsection{General Position and Jiggling}
By a {\bf polyhedron} $P \subseteq \R^n$ we will mean a finite union of compact {\bf convex} polyhedra.
\begin{defn}[Thurston \cite{Th}]
Let $\xi$ be any continuous distribution of codimension-$k$. Then an $n$-dimensional polyhedron $P \subseteq \R^n$ is in {\bf general position} with respect to $\xi$ if the projection $\R^n \longrightarrow \R^n/\xi_p$ is non-degenerate on all faces of dimension $k$ for every $p \in P$. 
\end{defn}
\noindent Note that the general position condition is $C^0$-open. We then have Thurston's jiggling lemma.
\begin{lem}[Jiggling Lemma \cite{Th}]\label{jiggling_lemma}
Let $\xi$ be any continuous distribution on a manifold $M$ with smooth triangulation $\mathcal{T}$. Then after subdividing and perturbing $\mathcal{T}$ we may assume that $\xi$ is in general position with respect to the new triangulation $\mathcal{T}'$. Moreover, each simplex can be assumed to be arbitrarily small.
\end{lem}
Given a polyhedron $P\subseteq \R^3$ in general position with respect to a plane field there are precisely two vertices $v_+,v_-$ called {\bf supporting vertices} so that $P$ is contained in one of the half spaces determined by the cooriented planes $\xi_{v_{\pm}}$. We denote by $v_+$ the vertex at the top and by $v_-$ the vertex at the bottom of $P$ with respect to the coorientation of $\xi$. Note that all these definitions are diffeomorphism invariant. Furthermore, if $P$ is in general position with respect to $\xi$, then we obtain a continuous characteristic line field on $\partial P$ away from the supporting vertices, which is bi-valued along the interiors of edges. If the original line field was the tangent distribution of a foliation, then this line field is integrable and the resulting foliation $\xi(\partial P)$ defines a (continuous) holonomy map for any transversal joining the two supporting vertices (which is trivial). For a general $C^0$-plane field $\xi$ this holonomy map may not be well-defined. However, more or less by definition, one can define the holonomy if $\xi$ is a $C^0$-confoliation (cf.\ Remark \ref{rem_hol} ff.).

By cleverly adding and removing small tetrahedra from the simplices in the triangulation obtained by jiggling, Colin \cite{Col} showed that there is a polyhedral decomposition where each vertex is a supporting for at most one polyhedron. It is important to note that after this modification the polyhedra are no longer tetrahedra and may not be convex. Moreover, these modified polyhedra do not meet nicely, in that two polyhedra may intersect only in part of a face or an edge (cf.\ Figure \ref{Colin_poly}). The advantage of Colin's notion of a polyhedral decomposition is that it allows constructions to be done near supporting vertices one polyhedron at a time.
\begin{lem}[\cite{Col}, Lemme 2.3]\label{lem_support}
Let $\xi$ be a smooth plane field on a closed $3$-manifold $M$. Then there is a polyhedral decomposition of $M$ such that $\xi$ is in general position with respect to each polyhedron and each vertex is supporting for at most one polyhedron. 

Moreover, there are precisely $3$ edges that meet at each supporting vertex and each skeleton is arbitrarily close the corresponding skeleton of the original triangulation.
\end{lem}
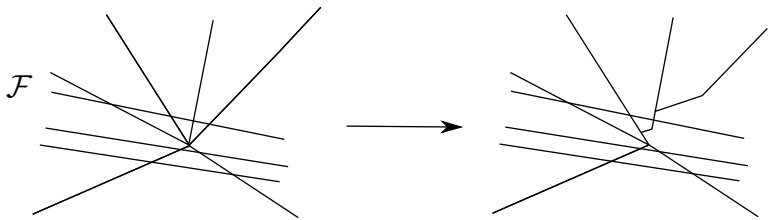
\begin{figure}[h]
\psset{xunit=.5pt,yunit=.5pt,runit=.5pt}
 \rput(-10,100){$\mathcal{F}$}  
\begin{pspicture}(555.76000664,160.00348789)
{
\newrgbcolor{curcolor}{0 0 0}
\pscustom[linewidth=1,linecolor=curcolor]
{
\newpath
\moveto(348.05739,3.72629789)
\lineto(466.30682,55.41227789)
}
}
{
\newrgbcolor{curcolor}{0 0 0}
\pscustom[linewidth=1,linecolor=curcolor]
{
\newpath
\moveto(465.59254,56.12657789)
\lineto(548.97323,1.13167789)
}
}
{
\newrgbcolor{curcolor}{0 0 0}
\pscustom[linewidth=1,linecolor=curcolor]
{
\newpath
\moveto(362.02111,96.12657789)
\lineto(538.44969,60.41228789)
}
}
{
\newrgbcolor{curcolor}{0 0 0}
\pscustom[linewidth=1,linecolor=curcolor]
{
\newpath
\moveto(357.7354,68.98370789)
\lineto(541.30683,39.69799789)
}
}
{
\newrgbcolor{curcolor}{0 0 0}
\pscustom[linewidth=1,linecolor=curcolor]
{
\newpath
\moveto(353.44969,56.12657789)
\lineto(534.87826,28.26942789)
}
}
{
\newrgbcolor{curcolor}{0 0 0}
\pscustom[linewidth=1,linecolor=curcolor]
{
\newpath
\moveto(348.05739,3.72629789)
\lineto(466.30682,55.41227789)
}
}
{
\newrgbcolor{curcolor}{0 0 0}
\pscustom[linewidth=1,linecolor=curcolor]
{
\newpath
\moveto(0.20025,3.01200789)
\lineto(118.44968,54.69798789)
}
}
{
\newrgbcolor{curcolor}{0 0 0}
\pscustom[linewidth=1,linecolor=curcolor]
{
\newpath
\moveto(118.99742,54.71642789)
\lineto(55.95854,154.10012789)
}
}
{
\newrgbcolor{curcolor}{0 0 0}
\pscustom[linewidth=1,linecolor=curcolor]
{
\newpath
\moveto(136.73414,149.96610789)
\lineto(117.82803,53.59666789)
}
}
{
\newrgbcolor{curcolor}{0 0 0}
\pscustom[linewidth=1,linecolor=curcolor]
{
\newpath
\moveto(118.19715,54.95052789)
\curveto(218.88788,158.90192789)(186.35647,126.13392789)(218.13026,159.65954789)
}
}
{
\newrgbcolor{curcolor}{0 0 0}
\pscustom[linewidth=1,linecolor=curcolor]
{
\newpath
\moveto(117.7354,55.41228789)
\lineto(201.11606,0.41738789)
}
}
{
\newrgbcolor{curcolor}{0 0 0}
\pscustom[linewidth=1,linecolor=curcolor]
{
\newpath
\moveto(14.16397,95.41228789)
\lineto(190.59256,59.69799789)
}
}
{
\newrgbcolor{curcolor}{0 0 0}
\pscustom[linewidth=1,linecolor=curcolor]
{
\newpath
\moveto(9.87826,68.26941789)
\lineto(193.44966,38.98370789)
}
}
{
\newrgbcolor{curcolor}{0 0 0}
\pscustom[linewidth=1,linecolor=curcolor]
{
\newpath
\moveto(5.59255,55.41228789)
\lineto(187.02116,27.55513789)
}
}
{
\newrgbcolor{curcolor}{0 0 0}
\pscustom[linewidth=1,linecolor=curcolor]
{
\newpath
\moveto(0.20025,3.01200789)
\lineto(118.44968,54.69798789)
}
}
{
\newrgbcolor{curcolor}{0 0 0}
\pscustom[linewidth=1,linecolor=curcolor]
{
\newpath
\moveto(237.7354,69.69799789)
\lineto(324.16397,68.98371789)
}
}
{
\newrgbcolor{curcolor}{0 0 0}
\pscustom[linestyle=none,fillstyle=solid,fillcolor=curcolor]
{
\newpath
\moveto(314.16431148,69.06635904)
\lineto(310.13139162,65.09955209)
\lineto(324.16397,68.98371789)
\lineto(310.19750454,73.09927891)
\lineto(314.16431148,69.06635904)
\closepath
}
}
{
\newrgbcolor{curcolor}{0 0 0}
\pscustom[linewidth=1,linecolor=curcolor]
{
\newpath
\moveto(314.16431148,69.06635904)
\lineto(310.13139162,65.09955209)
\lineto(324.16397,68.98371789)
\lineto(310.19750454,73.09927891)
\lineto(314.16431148,69.06635904)
\closepath
}
}
{
\newrgbcolor{curcolor}{0 0 0}
\pscustom[linewidth=1,linecolor=curcolor]
{
\newpath
\moveto(465.84262,55.73347789)
\lineto(361.5575,110.28285789)
}
}
{
\newrgbcolor{curcolor}{0 0 0}
\pscustom[linewidth=1,linecolor=curcolor]
{
\newpath
\moveto(118.19715,54.95052789)
\curveto(218.88788,158.90192789)(186.35647,126.13392789)(218.13026,159.65954789)
}
}
{
\newrgbcolor{curcolor}{0 0 0}
\pscustom[linewidth=1,linecolor=curcolor]
{
\newpath
\moveto(555.98346,143.66217789)
\curveto(541.93873,129.17205789)(526.41861,113.09387789)(506.76858,92.80767789)
\lineto(470.71429,80.71429789)
}
}
{
\newrgbcolor{curcolor}{0 0 0}
\pscustom[linewidth=1,linecolor=curcolor]
{
\newpath
\moveto(118.99742,54.71642789)
\lineto(55.95854,154.10012789)
}
}
{
\newrgbcolor{curcolor}{0 0 0}
\pscustom[linewidth=1,linecolor=curcolor]
{
\newpath
\moveto(466.85456,54.71642789)
\lineto(403.81568,154.10012789)
}
}
{
\newrgbcolor{curcolor}{0 0 0}
\pscustom[linewidth=1,linecolor=curcolor]
{
\newpath
\moveto(117.84262,55.73347789)
\lineto(13.5575,110.28285789)
}
}
{
\newrgbcolor{curcolor}{0 0 0}
\pscustom[linewidth=1,linecolor=curcolor]
{
\newpath
\moveto(484.73414,151.96610789)
\lineto(468.6283,67.44118789)
\lineto(460.71429,64.99999789)
}
}
\end{pspicture} \caption{A schematic picture illustrating Colin's procedure to ensure that each vertex is supporting for precisely one polyhedron.}\label{Colin_poly}
\end{figure}
At some points it will be convenient to have polyhedral decompositions that are in general position with respect to a line field say $Y$. In this case general position implies that $Y$ is transverse to all faces and is nowhere tangent to edges of any polyhedron $P$. The boundary of $P$ then decomposes into two discs $D_{in} \cup D_{out}$ that meet along a waist circle contained in the $1$-skeleton, so that each flow line of $Y$ points into $P$ on the interior of $D_{in}$, out of $P$ on the interior of $D_{out}$ and intersects $P$ in a point along $D_{in} \cap D_{out}$. 

\begin{rem}\label{tangent_conf}
Assume that $\xi$ is a $C^0$-confoliation in the tangential sense (Definition \ref{tangential_conf}) with respect to $Y$ smooth near the $1$-skeleton and assume that a polyhedral decomposition is in general position with respect to both $\xi$ and $Y$. If for each polyhedron the characteristic foliation on $\partial P$ is uniquely integrable, then the holonomy is non-positive and it is strictly negative for orbits passing through smooth contact regions. The reason being that we can think of $\xi$ as being tangent to a foliation by discs tangent to $Y$ on the front face of $P$ and negatively transverse to the back face, where and back and front faces correspond to those points of $\partial P$ where the vector field $Y$ points in, respectively out of $\partial P$.
\end{rem}

\subsection{Modifying plane fields to become contact near the $2$-skeleton}\label{modify_plane} 
For the moment we assume that $\xi$ is a smooth $2$-plane field in general position with respect to some polyhedral decomposition of $M$ as in Lemma \ref{lem_support}. We now describe explicitly how to modify such a plane field to a contact structure near the $2$-skeleton of any (polyhedral) subcomplex $K \subseteq M$ in a $C^0$-small fashion following Colin (cf.\ \cite{Col} Lemme 3.3). 

For each vertex $v$ choose a disc $D_v$ containing $v$ that is transverse to $\xi$ and choose a vector field $X_v$ that is tangent to $\xi$ and also transverse to $D_v$. This then gives a small neighbourhood $D_v \times [-1,1]$ containing $v$ in its interior so that the interval fibers are tangent to $\xi$. For example we can take $X_v$ to be the normalised vector field given by intersecting $\xi$ with a foliation by discs that is transverse to $D_v$ and say tangent to some normal line field. 

Next for each edge $e$ choose a thin strip $S_e = e \times [-\epsilon,\epsilon]$ containing $e$ that is transverse to $\xi$ and to $X_v$ near each vertex. Then extend $X_v$ to a vector field $X_e$ transverse to $S_e$ and tangent to $\xi$ which agrees with $X_v$ near the end points of $e$. This gives a small neighbourhood $S_e \times [-1,1]$ of the edge $e$ so that the intervals $pt \times [-1,1]$ are tangent to $\xi$. Finally for each face $f$ choose a vector field $X_f$ transverse to $f$ that is tangent to $\xi$, agrees with $X_e$ near edges and with $X_v$ near vertices. Flowing along this vector field then again gives a product neighbourhoood of $f$ so that the interval fibers are tangent to $\xi$. We next detail how one defines the desired plane field inductively over each skeleton.

\subsection*{Near the 0-skeleton:} We extend the plane field on $D_v \times [-1,1]$ by first taking the intersection with $D_v = D_v \times \{0\}$ and then taking a plane field given by twisting along the intervals tangent to $X_v$. Call this extension near each vertex $\hat{\xi}_0$. That is we consider coordinates $(x,y,z)$ so that $\frac{\partial}{\partial y}$ corresponds to $X_v$. Then 
$$\xi = Ker(dz-f(x,y,z)dx)$$
and we set $\hat{\xi}_0 = Ker(dz-\hat{f}(x,y,z)dx)$, where $\hat{f}(x,0,z)= f(x,0,z)$ and $\frac{\partial \hat{f}}{\partial y} >0$. 

\subsection*{Near the 1-skeleton:} We next extend near each strip $S_e$ by first intersecting $\xi$ with $S_e$ to obtain the characteristic foliation $\xi(S_e)$ and then twisting along the intervals tangent to $X_e$ in such a way that the resulting plane field $\hat{\xi}_1$ extends $\hat{\xi}_0$. By making the neighbourhoods of $S_e$ sufficiently small we can still assume that the resulting plane field is in general position.

\subsection*{Near the 2-skeleton:}  We finally extend $\hat{\xi}_1$ to a neighbourhood of $f$ by twisting along $X_f$ in such a way that the resulting plane field $\hat{\xi}_2$ is contact on some neighbourhood of the $2$-skeleton. 

\begin{rem}\label{close_plane_field}
Note that the resulting plane field can be assumed to be $\epsilon$-$C^0$-close to $\xi$ on (and hence near) the $2$-skeleton. Also for polyhedral decompositions in the sense of Colin, neighbourhoods of supporting vertices may meet the interiors of edges. This means that the modifications over the $1$-skeleton must be made relative to neighbourhoods of certain subsets, but this is automatically achieved by the construction above.
\end{rem}
\begin{rem}\label{polar_mod}
In the first step of the construction above, there is some freedom in defining $\hat{\xi}_0$. In particular, if $N_v = D^2 \times [-1,1]$ is a small neighbourhood of a supporting vertex $v$ so that the discs $D^2 \times \{pt\}$ are tangent to $\xi$ (in particular $\xi$ is integrable here), then we can assume that $\hat{\xi}_0 = Ker(dz -\delta \thinspace r^2 d\theta)$ where $(r,\theta)$ are polar coordinates on $D^2$ and $\delta > 0$ is small. We shall call the ordinary modification \emph{tangential} and the latter \emph{polar} near vertices. Note that the polar modification is really a tangential modification taken with respect to a particular choice of vector field so that the resulting plane field has a certain form relative to a fixed foliation near $v$.
\end{rem}

\noindent We record the following relative version of the above construction for later use: 
\begin{lem}\label{rem_relative_case}
Let $K$ be a finite collection of polyhedra in a polyhedral decomposition which is in general position with respect to a plane field $\xi$ and let $L \subseteq K$ be a subcomplex. Suppose that $\xi$ is a smooth plane field defined on an open neighbourhood of $K$. Suppose further that on some open neighbourhood $\mathcal{O}p(L)$ the plane field $\xi$ is contact. Then there is a contact structure $\hat{\xi}$ defined on $U_\epsilon = \mathcal{O}p(K^{(2)} \cup L)$ that is $\epsilon$-$C^0$-close to $\xi$ and agrees with $\xi$ on $V_{\epsilon} = K^{(2)} \setminus \mathcal{O}p(K^{(1)}) \cup \mathcal{O}p(L)$.

If in addition the restriction of $\xi$ to a given (fixed) neighbourhood $ \mathcal{O}p(K^{(1)})$ is a confoliation, then we can assume that $\hat{\xi}$ is defined on some slightly smaller $\mathcal{O}p'(K^{(1)}) \subset U_{\epsilon}$ independent of $\epsilon$.
\end{lem}
\begin{proof}
Recall that the vector fields $X_v,X_e,X_f$ constructed above are tangent to $\xi$ near the $2$-skeleton of $L$ so that the extension $\hat{\xi}$ can be chosen to agree with $\xi$ on $\mathcal{O}p'(L)\subseteq \mathcal{O}p(L)$. We can extend these plane fields over the interiors of polyhedra in $L$ by simply taking $\hat{\xi} = \xi$. Note that by construction for any polyhedron $P$ we can assume that $\hat{\xi}$ agrees with $\xi$ on $\partial P \setminus W_1$ for a small (but \emph{a priori} fixed) neighbourhood $W_1=\mathcal{O}p(K^{(1)})$. 

Finally under the assumption that $\xi$ is a confoliation on $W_1$ we can first make an $\epsilon$-close extension $\hat{\xi}$ which is a confoliation so that $\hat{\xi} = \xi$ on a slightly smaller neighbourhood of $K^{(1)}$ by assuming that both $X_v$ and $X_e$ and their flows are defined on $W'_1 \subset W_1$. This confoliation can then be deformed to a contact structure that is say $2\epsilon$-close to $\xi$.
\end{proof}
\begin{rem}\label{parametric_version}
The above lemma also holds parametrically: if $(\xi_t)_{t \in [0,\delta]}$ is a smooth family of plane fields so that $(\xi_t)_{t \in [0,\epsilon]}$ is contact on $\mathcal{O}p(L)$, then there is a smooth family of contact structures $(\hat{\xi}_t)_{t \in [0,\delta]}$ that is $\epsilon$-close to $(\xi_t)_{t \in [0,\delta]}$ and such that $\hat{\xi}_t = \xi_t $ on $K^{(2)} \setminus \mathcal{O}p(K^{(1)}) \cup \mathcal{O}p(L)$. 
\end{rem}

\subsection{Extending contact structures to the interior of a polyhedron}
At a certain point in our approximation scheme we will need to extend contact structures over $3$-cells in a $C^0$-controlled manner (cf.\ Lemma \ref{extend_close}). For this we first collect some preliminary lemmas:
\begin{lem}\label{extend_C^0}
Let $\mathcal{F}$ be a $C^0$-foliation on $U = D^2 \times [0,1]$ that agrees with the product foliation near the horizontal boundary $\partial^{h} U=D^2 \times \{0,1\}$ and is transverse to the interval fibers $\{pt \} \times [0,1]$. Suppose that $\mathcal{G}_n$ is a sequence of foliations by circles on the vertical boundary $\partial^{v} U = \partial D^2 \times [0,1]$ that agrees with the product foliation near the horizontal boundary and converges to the characteristic foliation $\mathcal{G}_0 = \mathcal{F}(\partial^v U) $ in the $C^0_{Fol}$-norm. Then there is a sequence of foliations $\mathcal{F}_n$ converging to $\mathcal{F}$ in the $C^0_{Fol}$-norm so that the characteristic foliations $\mathcal{F}_n(\partial^v U )$ are $\epsilon_n$-close to $\mathcal{G}_n$ in the $C^0_{Fol}$-norm for any sequence $(\epsilon_n)_{n\in\mathbb{N}}$ of positive numbers with $\epsilon_n \searrow 0$.

\end{lem}
\begin{proof}
Suppose that the foliation $\mathcal{F}$ is given as the graphs of a $C^0$-family of smooth functions 
$$f_z(x,y)\co D^2 \longrightarrow \R $$
and that the circle foliations $\mathcal{G}_n$ are given by a family of functions $g^n_z(\theta)\co \partial D^2 \longrightarrow \R$ so that $g^0_z = f_z|_{\partial D^2}$. Now for each $n$ pick $N \gg 1$ so that for any $\eta < 1/N$
$$\|f_z(x,y) -f_{z + \eta}(x,y)\|_{C^0_{Fol}} <1/n \textrm{ and } \|g_z(\theta) - g_{z + \eta}(\theta)\|_{C^0_{Fol}}  < \epsilon_n.$$
Then for any $z \in [(k+1)/N,k/N]$ we have
$$\left\|f_z(x,y) - \left(\rho_{k}(z)f_{(k-1)/N}(x,y) +(1-\rho_{k}(z)) f_{k/N}(x,y)\right)\right\|_{C^0_{Fol}} <1/n$$
for any positive bump function $\rho_k$ with $\text{supp}(\rho_{k}) = [(k+1)/N,k/N]$. Now let $X_k$ be a radial vector field tangent to the graph of $f_{k/N}$, i.e.\ to the leaf $L_k$ of $\mathcal{F}$ through $z=k/N$. We extend this vector field slightly to a (possibly very) small open neighbourhood of $L_k$. Then for $n$ sufficiently large we can extend $g^n_{k/N}$ first to a map of the annulus $S^1 \times [1/2,1]$ using $X_k$ and then to a map of the the disc that agrees with $f_{k/N}$ away from the boundary by using a fixed cut-off function. For $n$ large the resulting smooth map $\widehat{f}^n_{k/N}$ is close to $f_{k/N}$ in the $C^0_{Fol}$-norm. We then simply use linear interpolation in the $z$ parameter to extend the maps $\widehat{f}^n_{k/N}$ to obtain a family of maps $$\widehat{f}^n_z(x,y)\co D^2 \longrightarrow \R. $$
Note that by taking $N \gg 1$ we can assume that the foliation $\mathcal{F}$ is smooth for $z \in [0,k/2N] \cup [1-1/2N,1] = I_0 \cup I_1$ and thus we can also assume that $\widehat{f}^n_z(x,y) = f^n_z(x,y) = f_z(x,y)$ for $z \in I_0 \cup I_1$. This family has the desired properties, except that the map 
$$F^n(x,y,z) = (x,y,\widehat{f}^n_z(x,y))$$
is only a smooth bijection and $\frac{\partial F^n}{\partial z}$ is in general only non-negative (but strictly positive near $z = 0,1$). This can be remedied by taking some bump function $\rho$ that is identically $1$ on $[1/2N,1-1/2N]$ and has support in the interior of $[0,1]$ and setting
$$\tilde{F}^n(x,y,z) = (x,y,\widehat{f}^n_z(x,y) + \delta_n \cdot\rho(z))$$
for $\delta_n$ sufficiently small. 
\end{proof}
\begin{lem}\label{polar}
Let $(\xi_t)_{t \in [0,\delta]}$ be a $C^0$-family of plane fields defined near the $2$-skeleton of a polyhedral decomposition $K$. Suppose that $\xi_t$ is in general position with respect to the polyhedral decomposition and let $v$ be a supporting vertex. Assume furthermore that the family is smooth on $U = \mathcal{O}p(v)$, is integrable for $t=0$ and contact for $t>0$. Then there is a family $(\eta_t)_{t \in [0,\delta]}$ on $\mathcal{O}p(K^{(2)})$ agreeing with $\xi_t$ away from $U \cap \mathcal{O}p(K^{(2)})$ such that $\eta_0 = \xi_0$ and each $\xi_t$  is polar on $U' \subset U$ for all $t>0$.
\end{lem}
\begin{proof}
Let $(x,y,z)$ be foliated coordinates on $U = [0,\epsilon]\times[0,1] \times [0,\epsilon] $ so that $\xi_ 0 = Ker(dz)$. We will want to assume that $U$ is very thin in the $(x,z)$-plane so that the discs $y = 0,1$ do not meet the $2$-skeleton. Let $X_v(t)$ be a smooth family of vector field tangent to $\xi_t$ on $U$. Define a smooth family of plane fields on $U$ by $\hat{\xi}_t= Ker(dz -\delta(t) \thinspace r^2 d\theta)$ for some smooth function with $\delta(t) > 0$ for $t>0$ and $\delta(0)=0$. We then straighten things out by a smooth family of isotopies $\Phi_t$ with $\Phi_0 = id$ with support near $v$ so that all plane fields are tangent $X_v(t)$. Thus in suitable coordinates $(x_t,y_t,z_t)$ (depending on $t$) we have
$$\xi_t = Ker(dz_t-a_t(x_t,y_t,z_t)dx_t) \quad , \quad \Phi_t^*\thinspace \hat{\xi}_t = Ker(dz_t-b_t(x_t,y_t,z_t)dx_t),$$
where the contact condition is that the functions $a_t,b_t$ are monotone in the $y$-direction. By using a partition of unity in the $(x,z)$-plane we can modify $(\xi_t)_{t\in[0,\delta]}$ to agree with $\Phi_t^*\thinspace \hat{\xi}_t$ on some (fixed) neighbourhood of $v$, since the condition of being monotone is convex. We then push forward under $\Phi_t$ and extend by $\xi_t$ to obtain $\eta_t$.
\end{proof}
The following technical result will allow us to modify plane fields to contact structures with $C^0$-control and is a slight sharpening of the corresponding results of Colin (\cite{Col} Lemme 4.1) and Vogel (\cite{Vog} Lemma 4.14) that is specifically adapted to our needs. 
\begin{lem}\label{extend_close}
Let $\mathcal{F}$ be a $C^0$-foliation and consider a polyhedral decomposition as in Lemma \ref{lem_support} with respect to which $\mathcal{F}$ is in general position. Now let $K \subseteq M$ be any collection of polyhedra and let $(\xi_t)_{t \in [0,\delta]}$ be a continuous family of $C^0$-confoliations on some neighbourhood of $K$ so that 
\begin{enumerate}
\item $\xi_0$ is tangent to $\mathcal{F}$; 
\item The family $\xi_t$ is smooth on some (fixed) neighbourhood $U_1$ of the $1$-skeleton;
\item Each $\xi_t$ is a contact structure on some (fixed) neighbourhood of each supporting vertex for all $t>0$;
\item For each polyhedron $P$ there are characteristic foliations tangent to the characteristic line field $\xi_t(\partial P)$ that vary continuously in $t$ with respect to the $C^0_{Fol}$-norm. 
\item All leaves of these characteristic foliations pass through a contact region contained in $\partial P \setminus V_1$, for some neighbourhood of the $1$-skeleton $V_1 \subseteq U_1$. 
\end{enumerate}
Then there is a contact structure on a neighbourhood of $K$ which is $\epsilon$-$C^0$-close to $\mathcal{F}$ for any $\epsilon >0$. 

Moreover, if the family $(\xi_t)_{t \in [0,\delta]}$ is smooth on an open neighbourhood of a subcomplex $L \subseteq K$ and contact (near $L$) for $t>0$, then we can assume that this contact structure agrees with $\xi_{t_0}$ near $L$ for some $0<t_0\ll 1$.
\end{lem}
\begin{rem}
Note that the condition (5) above implies that the holonomy of $\xi_t(\partial P)$ is negative away from supporting vertices for all $t > 0$. Also although the holonomy depends \emph{a priori} on the choice of characteristic foliation, by abuse we shall refer to $\xi_t(\partial P)$ simply as the characteristic foliation on the boundary of a polyhedron $ P$.
\end{rem}
\begin{proof}
Let $P$ be any polyhedron of $K$ and choose a smoothing $P_0 \subseteq P$ that agrees with $P$ away from $V_1$. Fix small foliated (with respect to $\xi_0$) neighbourhoods $N_P^{\pm} \cong D^2 \times [-\epsilon,\epsilon] \subseteq U_1$ of each supporting vertex so that the discs are tangent to $\xi_0$, the interval $0 \times [-\epsilon,\epsilon]$ points into respectively out of $P$ and the supporting vertices correspond to $(0,0)$. We attach $N_P^{\pm}$ to $P_0$ and smoothen to obtain a subset $P'$ which is diffeomorphic to $ D^2 \times [0,1]$ (as a smooth manifold with corners). We can assume that the discs are tangent to $\xi_0$ near the top and bottom of $P'$. We can also assume that the tangent distribution to the discs $D^2 \times pt$ themselves is $C^0$-close to $T \mathcal{F}$ and that the intervals $ pt \times [0,1]$ are transverse to $\xi_t$. Thus we can define the holonomy of the characteristic foliation on the vertical boundary $\partial^v P' = S^1 \times [0,1]$. With respect to the product coordinates on $P'$ the foliation $\mathcal{F}$ is given by the graphs of functions $f_z(r,\theta) \co D^2 \longrightarrow \R$, where $(r,\theta)$ are polar coordinates on $D^2$.

Now the condition that $\xi_t$ is a $C^0$-confoliation and the fact that every leaf of $\xi_t(\partial P)$ passes through the contact region on $\partial P \setminus V_1$ for $t>0$ implies that the same holds for $\xi_t(\partial^v P')$. It follows that the holonomy $h^t_P$ of $\xi_t(\partial^v P')$ is strictly negative and $h^t_P < -\delta_P(t)<0$ for some continuous function $\delta_P(t)$ with $\delta_P(0) = 0$. Here we may have to extend $P'$ a little to make sure that the holonomy is well-defined. 

We can then approximate $\xi_t$ by a $C^0$-family of plane fields $\zeta_t$ that are smooth for $t>0$, satisfy $\zeta_0 = \xi_0$ and agree with $\xi_t $ on $V_1 \cup \mathcal{O}p(L)$. To do this in such a way that the holonomy is controlled some additional care is needed. First for a face $F$ consider the induced foliation $\xi_t(F)$, which is smooth on $\mathcal{O}p(\partial F)$. By general position each of the leaves is an arc and these arcs can be parametrised as a continuous family of smooth arcs. One can then smoothen as in Lemma \ref{smooth} to obtain a smooth line field on $F \setminus \mathcal{O}p(\partial F)$ and using a bump function this gives a $C^0$-approximation on $F$, which can be extended to a small neighbourhood by flowing along a vector field that is sufficiently close to $\xi_t$. Note that this approximation can be done parametrically. By making this approximation sufficiently close (parametrically) to $\xi_t$ we can continue to assume the holonomy of $\zeta_t(\partial^v P')$ is negative for $t>0$. As our smoothing was done via foliations, rather than merely smoothing plane fields we can assume that the family $\zeta_t(\partial^v P')$ is continuous in the $C_{Fol}^0$-norm and that $\zeta_0(\partial^v P') = \mathcal{F}(\partial^v P')$. After this one can smoothen the plane field on the interior of each $3$-cell (as needed) relative to some neighbourhood $\mathcal{O}p(K^{(2)})$.

Let $\epsilon_t > 0 $ be such that $\epsilon_t \searrow 0$. For each $t > 0$ sufficiently small we now apply Lemma \ref{rem_relative_case} to obtain contact structures $\hat{\zeta}_t$ that are $\epsilon_t$-$C^0$-close to $\zeta_t$ on $\mathcal{O}p_t(K^{(2)})  \cup\mathcal{O}p(K^{(1)})$ (where the size of the first neighbourhood possibly depends on $t$) and so that $\hat{\zeta}_t = \zeta_t$ on $\mathcal{O}p(L)$ and near all supporting vertices of $K$. Moreover, we can still assume that the holonomies of the characteristic foliations $\hat{\zeta}_t(\partial^v P')$ are negative for all $t>0$ by taking $\epsilon_t$ sufficiently small.



We must now extend the contact structure $\hat{\zeta}_t$ over polyhedra $P$ that are not wholly contained in the neighbourhood $\mathcal{O}p'(L) \subseteq \mathcal{O}p(L)$ given by Lemma \ref{rem_relative_case}. We let $P$ be such a polyhedron. By Lemma \ref{polar} we can assume that $\hat{\zeta}_t$ is given via a polar modification near supporting vertices, i.e.\ as the kernel of $dz -\delta(t)r^2 d \theta$ for polar coordinates $(r,\theta)$ on the $D^2$-factor of $P'$. 

The argument now proceeds in 3 steps: we first find a continuous family of $C^0$-foliations by circles $\mathcal{G}_t$ on $\partial P'$ so that for $t>0$ each $\mathcal{G}_t$ is smooth and transverse to $\hat{\zeta}_t$ and $\mathcal{G}_0$ agrees with the foliation by circles induced by $\mathcal{F}$. For a sequence $t_n \searrow 0$ we then take a sequence of  smooth foliations $\mathcal{F}_{t_n}$ converging to $\mathcal{F}|_{P'}$ such that $\mathcal{F}_{t_n} = \mathcal{F}$ near the top and bottom of $P'$ and most importantly the foliations are transverse to $\hat{\zeta}_{t_n}$ on $\partial P'$. Finally one extends over $P'$ by twisting along Legendrian curves tangent to leaves of $\mathcal{F}_{t_n}$. The resulting contact structure will again be $C^0$-close to $\mathcal{F}_{t_n}$, and hence to $\mathcal{F}$, and agrees with $\hat{\zeta}_{t_n}$ near the $2$-skeleton. 

We now fill in some details:

\textbf{Step 1:} Let $X_t$ be the normalised vector field that is tangent to the (oriented) characteristic foliation $\hat{\zeta}_t( \partial P')$. Note that by our choice of smoothing above the orbits of the flow generated by $X_t$ converge to the leaves of the characteristic foliation $\zeta_0(\partial^v P') = \mathcal{F}(\partial^v P')$ in the $C^0_{Fol}$-sense. Consider the family of vector fields on $\partial D^2 \times [0,1]$ given by
$$Y_{t} = X_t + \delta(t) \cdot \partial_z,$$
where $z$ denotes the second coordinate in $D^2 \times [0,1]$ and $\delta(t)$ is a continuous function that is positive for $t >0$ and $\delta(0)= 0$. Note that we can assume that this vector field remains transverse to the boundaries of the disc leaves near the top and bottom of $D^2 \times [0,1]$ by choosing $\delta(t)$ sufficiently small. 

We have now defined a continuous family of $C^0$-vector fields (smooth for $t>0$) each of which is everywhere (positively) transverse to $\hat{\zeta}_t(\partial P')$ for $t>0$. We can then modify the vector field using a partition of unity so that it becomes tangent to $\partial D^2 \times \{t\}$ for $t$ close to $0$ and $1$ respectively and remains transverse to $X_t$. By taking $\delta(t)$ sufficiently small we can still assume that the holonomies given by the flow along $Y_t$ are strictly negative for $t>0$.

Since the holonomies given by $Y_t$ on $P'$ are negative for $t > 0$, we can then modify the flow of $Y_t$ (parametrically) on a small strip $S \subseteq \partial P'$ contained in $V_1$ so that  all flow lines are closed circles and the flow remain transverse to $\hat{\zeta}_t(\partial P')$. We denote the resulting family of circle foliations by $\mathcal{G}_t$ and note that $\mathcal{G}_0$ agrees with the foliation on $\partial^v P'$ induced by $\mathcal{F}$.

\textbf{Step 2:} By Lemma \ref{extend_C^0} there is a sequence of smooth foliations $\mathcal{F}_n$ converging to $\mathcal{F}$ on $P'$ in the $C^0_{Fol}$-norm so that on $\partial P'$ the induced foliation is arbitrarily close to $\mathcal{G}_{t_n}$ and hence is transverse to $\hat{\zeta}_{t_n}$ for some sequence with $t_n \searrow 0$.

\textbf{Step 3:} 
We extend the contact structures $\hat{\zeta}_{t_n}$ to the interior of $P$ by twisting along the normalised vector field spanning the intersection $\hat{\zeta}_{t_n} \cap \mathcal{F}_{n}$. This twisting should be done very quickly near the boundary, so that the resulting plane field stays close to leaves of $\mathcal{F}_{n}$ on $P'$ -- the point being that this twisting reduces the angle between the plane field and $T\mathcal{F}_{n}$. One then twists very slowly into the center of each disc, which yields a contact structure that is close to $\mathcal{F}_{n}$ and hence to $\mathcal{F}$ for $n$ sufficiently large.  Note that this can be done relative to the neighbourhoods $N^{\pm}_P$ of supporting vertices as the contact structures are polar here by assumption and we thus obtain the desired extension of $\hat{\zeta}_{t_n}$ to a neighbourhood of all of $K$.
\end{proof}

\begin{rem}
Although Lemma \ref{extend_close} is only stated for $C^0$-confoliations it is also true for $C^0$-confoliations in the tangential sense as long as the the characteristic line field can be coherently integrated so that the holonomy of this characteristic foliation on $\partial P'$ is well defined. The reason for this is that the holonomies will be negative in this case too (cf.\ Remark \ref{tangent_conf}).
\end{rem}
\section{Smoothing Ribbons}\label{sec_ribbons}
We consider a polyhedral decomposition that is in general position with respect to a foliation $\mathcal{F}$ and a line field $X$ transverse to $\xi = T\mathcal{F}$. Let $A_i$ be a finite collection of annular fences tangent to $X$ so that any point in $M$ can be joined to the interior of $A_i$ by a smooth leafwise arc. We then have the following definition as in \cite{Vog}.
\begin{defn}
Consider a polyhedral decomposition that is in general position with respect to a $C^0$-foliation $\mathcal{F}$ and let $X$ be a smooth normal line field. 
A {\bf system of smoothing ribbons} adapted to $X$ is a finite collection of pairwise {\bf disjoint} smoothly embedded strips $R_j = \sigma_j \times [-\epsilon,\epsilon]$ such that 
\begin{enumerate}
\item The arcs $\sigma_j \times \{\pm \epsilon\}$ are tangent to the foliation $\mathcal{F}$, each ribbon is transverse to $\mathcal{F}$ and tangent to $X$;
\item For the initial point $p_j$ of $\sigma_j$ the interval $p_j \times [-\epsilon,\epsilon]$ lies on the boundary of some polyhedron $P$ and for the end point $q_j$ the interval $q_j \times [-\epsilon,\epsilon]$ lies in the interior of some $A_i$ and is tangent to $X$;
\item The system is {\bf full} if for each polyhedron $P$, any point of $\partial P$ can be joined to some $R_j$ by a curve tangent to the characteristic foliation $\mathcal{F}(\partial P)$ for all polyhedra outside a small neighbourhood of $A_i$. 
\item The ribbons intersect all polyhedra transversely and are disjoint from the $1$-skeleton except possibly at supporting vertices.
\end{enumerate}
\end{defn}
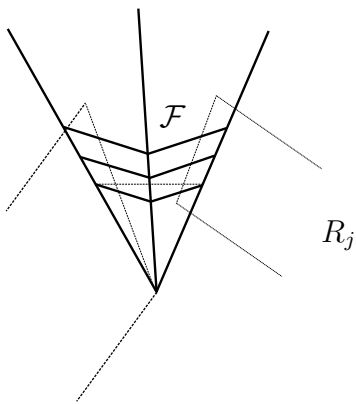
\begin{figure}[h]\label{Graph}
\psset{xunit=.9pt,yunit=.9pt,runit=.9pt}
\begin{pspicture}(132.90035969,165.35955264)
 \rput(70,120){$\mathcal{F}$}  \rput(140,70){$R_j$}
{
\newrgbcolor{curcolor}{0 0 0}
\pscustom[linewidth=1,linecolor=curcolor]
{
\newpath
\moveto(0.88986,156.78207264)
\lineto(63.39542,46.21111264)
\lineto(63.39542,46.21111264)
\lineto(110.54252,155.91490264)
}
}
{
\newrgbcolor{curcolor}{0 0 0}
\pscustom[linewidth=1,linecolor=curcolor]
{
\newpath
\moveto(55.76732,165.32740264)
\lineto(63.39542,46.96129264)
}
}
{
\newrgbcolor{curcolor}{0 0 0}
\pscustom[linewidth=0.3758944,linecolor=curcolor,linestyle=dashed,dash=0.7517888 0.37589441]
{
\newpath
\moveto(33.7511,125.06288264)
\lineto(63.05437,46.13858264)
}
}
{
\newrgbcolor{curcolor}{0 0 0}
\pscustom[linewidth=0.56384164,linecolor=curcolor,linestyle=dashed,dash=1.12768324 0.56384161]
{
\newpath
\moveto(0.22736,80.16670264)
\lineto(33.57247,125.64850264)
}
}
{
\newrgbcolor{curcolor}{0 0 0}
\pscustom[linewidth=0.39346445,linecolor=curcolor,linestyle=dashed,dash=0.78692889 0.39346445]
{
\newpath
\moveto(71.69887,83.22423264)
\lineto(88.5881,128.70662264)
}
}
{
\newrgbcolor{curcolor}{0 0 0}
\pscustom[linewidth=0.24446867,linecolor=curcolor,linestyle=dashed,dash=0.48893736 0.24446868]
{
\newpath
\moveto(71.71488,83.54931264)
\lineto(116.12106,52.66390264)
}
}
{
\newrgbcolor{curcolor}{0 0 0}
\pscustom[linewidth=1,linecolor=curcolor]
{
\newpath
\moveto(37.60986,91.47604264)
\lineto(61.0263,84.12971264)
\lineto(82.60615,91.01690264)
}
}
{
\newrgbcolor{curcolor}{0 0 0}
\pscustom[linewidth=0.40000001,linecolor=curcolor,linestyle=dashed,dash=0.8 0.4]
{
\newpath
\moveto(38.29858,91.47604264)
\lineto(82.60615,91.47604264)
}
}
{
\newrgbcolor{curcolor}{0 0 0}
\pscustom[linewidth=1,linecolor=curcolor]
{
\newpath
\moveto(31.18182,102.96826264)
\lineto(60.68194,94.12971264)
\lineto(87.65675,103.42742264)
}
}
{
\newrgbcolor{curcolor}{0 0 0}
\pscustom[linewidth=1,linecolor=curcolor]
{
\newpath
\moveto(24.17985,115.60836264)
\lineto(59.87844,104.24450264)
\lineto(92.70736,115.03443264)
}
}
{
\newrgbcolor{curcolor}{0 0 0}
\pscustom[linewidth=0.56384164,linecolor=curcolor,linestyle=dashed,dash=1.12768324 0.56384161]
{
\newpath
\moveto(29.86942,0.16670264)
\lineto(63.21453,45.64850264)
}
}
{
\newrgbcolor{curcolor}{0 0 0}
\pscustom[linewidth=0.24446867,linecolor=curcolor,linestyle=dashed,dash=0.48893736 0.24446868]
{
\newpath
\moveto(88.42438,128.80164264)
\lineto(132.83056,97.91623264)
}
}
\end{pspicture} \caption{Ribbons near a supporting vertex of a polyhedron $P$.}\label{Ribbons}
\end{figure}
Such a collection of ribbons is easy to construct by compactness (cf.\ \cite{Vog} Lemma 4.20). First choose paths that join all supporting vertices to some $A_i$ and small transversals tangent to $X$ near each vertex so that pushing along the leaves of $\mathcal{F}$ maps this transversal into the interior $A_i$. Then for a given polyhedron $P$ we consider a total transversal $\tau_P$ for the characteristic foliation $\xi(\partial P)$ consisting of finitely many disjoint compact intervals that do not meet the $1$-skeleton except at supporting vertices. We join each point of $\tau_P$ to some $A_i$ by a leafwise path. We then consider small transversals that are also mapped to $A_i$ under the holonomy along the chosen path. By compactness finitely many such transversals will cover $\tau_P$ away from neighbourhoods of the supporting vertices. By general position we may assume that the intersections with the boundary of each polyhedron are disjoint from the $1$-skeleton.

These ribbons may intersect and we now explain how to resolve these intersections. After a small isotopy we may assume that these intersections are transverse and we may also assume that all ribbons are disjoint near their endpoints. Then we inductively remove all intersections as follows: if two ribbons $R_1 = \sigma_1 \times [0,1]$ and $R_2 = \sigma_2 \times [0,1]$ intersect along an interval $J$, then we replace the second ribbon by one that is parallel to a sub-ribbon of $R_2$ and then runs parallel to $R_1$ and a second (possibly empty) ribbon that is a subset of $R_2$ (cf.\ Figure \ref{Ribbons_resolve}). Iterating this resolution process gives the desired collection of ribbons.
\begin{figure}[h]\label{resolution_of_ribbons}
\includegraphics[scale=0.7]{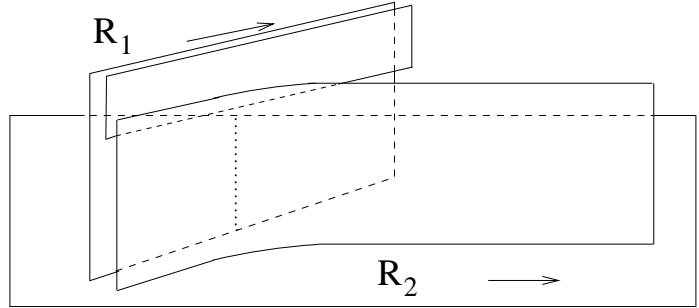} \caption{Resolving the intersection of two ribbons (picture courtesy of T. Vogel).}\label{Ribbons_resolve}
\end{figure}

These smoothing ribbons are an essential technical tool as they will allow us to smoothen our foliation near sufficiently many disjoint arcs so that we will be able to transport contactness around the manifold beginning with a confoliation that is only of class $C^0$.

\section{Proof of Theorem \ref{approx_cont}}\label{sec_proof}
The argument will consist of a finite sequence of approximations. Since each of these approximations can be made arbitrarily $C^0$-small at each step, this yields the required approximation. As outlined in the Introduction, the proof is carried out in 4 steps:

\subsection*{Step 1: Making the number of compact leaves finite}
We first modify $\mathcal{F}$ so that it only has finitely many compact leaves. This is identical to the smooth case in \cite{ETh} and is based on the following (cf.\ \cite{CC} Theorem 6.1.1):
\begin{thm}[Haefliger \cite{Hae}]
The set of compact leaves of a codimension-$1$ foliation on a closed manifold is compact.
\end{thm}
\noindent Note that by Reeb stability no compact leaf can be simply connected, as otherwise the foliation would be the product foliation on $S^2 \times S^1$, which has been excluded by assumption. In view of this there is a finite collection of (smooth) embeddings $N_k = \Sigma_k \times [0,c_k]$, whereby $\Sigma_k$ is a compact surface of genus at least $1$, so that each $N_k$ is a foliated $I$-bundle, $\Sigma_k \times \{0,c_k\}$ are leaves of $\mathcal{F}$ (we allow the possibility $c_k=0$) and any closed leaf is contained in some $N_k$. By subdividing we can assume that these foliated bundles are arbitrarily thin. We then insert a suspension foliation that does not have any closed leaves in the interior. By making the pieces we insert sufficiently thin this can be achieved in a $C^0$-small fashion. To be precise a closed leaf is given as the graph of a smooth function $f_{L}\co \Sigma_k \longrightarrow [0,1]$ and given two leaves $L_0,L_1$ the linear interpolation 
$$f = (1-t)f_{L_0} +t f_{L_1}$$
gives a smooth map of the product $\Sigma_k \times [0,1]$ sending top and bottom leaf to $L_0,L_1$ respectively. Since the partial $C^0_{Fol}$-norms of the maps are close this will also be true for the linear interpolation and the resulting foliation will be close to $\mathcal{F}$. One can then insert a (smooth) suspension foliation in a $C^0$-small fashion.

Note that after this modification each closed leaf is isolated. Moreover, any isolated closed leaf contains an embedded curve whose holonomy is non-trivial on both sides. It is important to remark that we do not claim that this holonomy is either (sometimes) attractive or repelling as the leaf may be unstable. 
\subsection*{Step 2: Producing holonomy near minimal sets}
We now manufacture finitely many embedded annuli $A_i$ transverse to $\mathcal{F}$ so that every point in $M$ can be connected to the interior of some $A_i$ by a curve that is contained in a leaf. We will call such a collection of annuli a {\bf transitive collection of annular fences}. We will also want that the cores $\gamma_i$ of these annuli are contained in a leaf and that the holonomy around $\gamma_i$ has a contracting/repelling interval, which will be achieved by blowing up leaves and inserting holonomy.

First observe that $\mathcal{F}$ now has only finitely many minimal sets, since by Lemma \ref{finite_exceptional} there are only finitely many exceptional minimal sets in general and by construction we have modified $\mathcal{F}$ to have only finitely many compact leaves. Note also that the closure of each leaf contains one of these minimal sets. This means that we need one annular fence $A_i$ for each minimal set to obtain a transitive collection. 

We now modify $\mathcal{F}$ so that these fences can be taken to have contracting/repelling holonomy. First approximate $\mathcal{F}$ so that it is very nice on neighbourhoods $N(A_i)$ of each annular fence $A_i$ (cf.\ Section \ref{nice_section}). By a further $C^{0}$-small perturbation these fences can be taken to be nice at the boundary.

Consider a minimal set $M_*$ (if $\mathcal{F}$ is minimal, then $M_* = M$). By assumption $M_*$ is a closed saturated subset containing at least one leaf, say $L_0$, which is not simply connected. Let $\gamma$ be a homotopically non-trivial embedded curve in $L_0$ and consider the holonomy around $\gamma$ for a sufficiently small normal arc, so that the holonomy map 
$$h\co I = (-\epsilon,\epsilon) \longrightarrow h(I)$$
is well defined and $h(0) =0$ corresponds to the curve $\gamma$. We consider several cases:

\bigskip

\underline{Case 1:} $h$ has a unique fixed point

\medskip

If the holonomy has no fixed points apart from $0$, then it is either attractive/repelling or attractive on one side and repelling on the other. In the former case we are done and in the latter case we blow up $L_0$ and using Lemma \ref{hol_insertion} we insert a suspension foliation so that $L_0$ is replaced by two leaves both of which have an attractive/repelling interval. In this way we replace one annular fence by two fences (cf.\ Figure \ref{hol_unstable}), one of which is contracting the other expanding (with respect to the induced orientation).

\bigskip

\begin{figure}[h]\label{Insert_hol_unstable}
\input{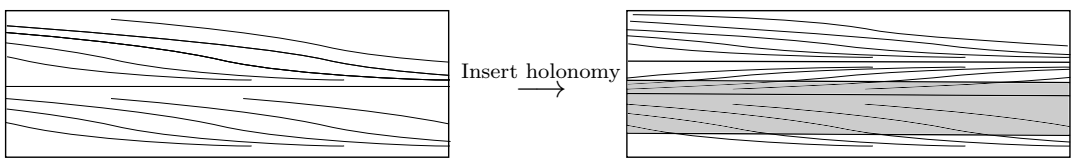} \caption{Inserting holonomy in the unstable case. One of the resulting fences is shaded grey.}\label{hol_unstable}
\end{figure}

\underline{Case 2:} $h$ is (germinally) non-trivial on each side of $L_0$

\medskip

Since the hypothesis implies that there are arbitrarily small contracting/repelling intervals on either side of $L_0$, we can argue exactly as in the previous case to obtain (arbitrarily thin) fences with attracting/repelling intervals.

\bigskip

\underline{Case 3:} $h$ has an interval of fixed points containing $0$

\medskip

In this case the holonomy is trivial on one or both sides of $L_0$. We consider a leaf $L'$ containing a curve $\gamma'$ parallel to the core of the annular fence $\gamma$ which is homotopically non-trivial in $L_0$ by assumption. We can assume that $\gamma'$ is again homotopically non-trivial in $L'$, unless $L_0$ is a closed toral leaf by Proposition \ref{nov_closed}. But since we have assumed that there are only finitely many closed leaves, there must be some embedded curve on $L_0$ with holonomy on both sides, so that we are in the previous case. Thus we can assume that $\gamma'$ is not homotopically trivial.

\begin{figure}[h]\label{Insert_hol}
\input{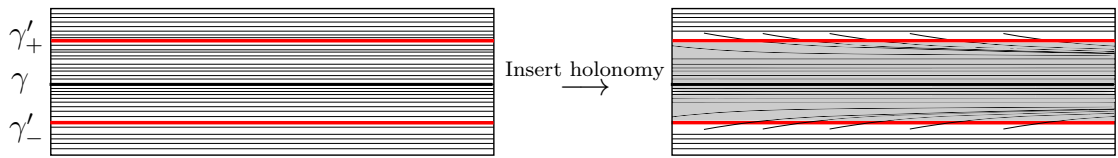} \caption{Inserting holonomy above and below the thickened black curve. The resulting fence is shaded grey.}\label{hol}
\end{figure}

We then thicken $L'$ and insert a suspension foliation with holonomy around $\gamma'$. If the holonomy is trivial on both sides then we thicken a leaf containing parallel copy of $\gamma$ both above and below $L$. We denote these curves by $\gamma'_{\pm}$. This then yields an annular fence with attracting/repelling holonomy (see Figure \ref{hol}). 

Since the resulting plane fields can be assumed to be $C^0$-close to the tangent plane field of $\mathcal{F}$, we can assume that the initial choice of annular fences remains transitive. We can also assume that all annular fences are very nice with respect to our original choice of nice coordinates and that their heights are arbitrarily small. 

\subsection*{Step 3: Producing a contact structure away from the annular fences: smoothing and transporting}

\

\bigskip

\noindent \textbf{Pick a triangulation and jiggle it:} Choose a triangulation of $M$ and let $A_1,\ldots,A_n$ be the collection of annular fences constructed above. Consider neighbourhoods $N(A'_i) \subseteq N(A_i)$ as in the proof of Lemma \ref{cont_holonomy}. We then apply Thurston's Jiggling Lemma (Lemma \ref{jiggling_lemma}) so that the triangulation is in general position with respect to $\mathcal{F}$. We will also require that the triangulation is in general position with respect to the plane field given by $dx = 0$ as well as the line field $\frac{\partial}{\partial y}$ associated to the chosen very nice coordinates $(x,y,z)$ on $N(A_i)$. It will also be convenient to assume that the triangulation is in general position with respect to the line field given by $\xi \cap \eta_y$ near $A_i$, where $\eta_y = Ker(dy)$.

We now modify the triangulation to a polyhedral decomposition as given by Lemma \ref{lem_support}.  By taking the initial triangulation sufficiently fine, we can assume that $\partial N(A_i)$ is contained in the interior of a subcollection $L$ of the resulting polyhedral decomposition that is disjoint from $N(A'_i)$.

\bigskip

\noindent \textbf{Smoothen near the $1$-skeleton:} After an initial perturbation we can assume that the foliation is smooth near $\partial A_i$ (Remark \ref{rem_nice_fences}). Next smoothen $\mathcal{F}$ on a neighbourhood $U_1$ of the $1$-skeleton as in Lemma \ref{smooth_1_skel}. We claim that the smoothened foliation $\mathcal{F}'$ can be constructed to be tangent to the $y$-intervals of the chosen very nice coordinates on $N(A_i)$. To achieve this consider the image $D_e$ of each edge $e$ that meets $N(A_i)$ under the flow in the $y$-direction. We then smoothen near each $D_e$, which preserves the property of being constant in $y$, and take the cut-off in Lemma \ref{smooth_1_skel}  to be constant in the $y$ direction on $N(A_i)$. Note that the discs $D_e$ may intersect, but by general position we can assume that these intersections are a collection of intervals, and smoothing near these intervals first, means that entire smoothing can be done coherently. Note that this procedure amounts to smoothing the characteristic foliation on the annulus $A$ near a collection of transverse arcs. These arcs may intersect, but these intersections can be assumed to be transverse, and smoothing near such a collection of arcs is not a problem. 
\

\noindent \textbf{Choose a full collection of ribbons:} Let $R_j$ be a full collection of ribbons for all polyhedra meeting the complement of the interiors of the neighbourhoods $N(A_i)$. We choose these so that the intersection with $N(A_i)$ is tangent to the $y$-intervals of the chosen very nice coordinates. This is possible by the assumption that the triangulation is in general position with respect to the plane field $dx = 0$. By a further perturbation we can also assume that the vertical boundary $\partial^v N(A_i)$ is transverse to the $2$-skeleton of each polyhedron $P$. Here we decompose the boundary $\partial N(A_i) = \partial ^v N(A_i) \cup \partial ^h N(A_i)$ into a vertical and a horizontal part (see Figure \ref{leaf_box}).

We next smoothen on small disjoint neighbourhoods of the ribbons $R_j$. Note that the neighbourhood of the $1$-skeleton $U_1$ above can be assumed disjoint from each ribbon except near supporting vertices, where the foliation has already been made smooth. Ribbons $R_j$ beginning at points in $\partial ^h N(A_i)$ can be taken to be wholly contained in the smooth region of $\mathcal{F}$ already, so that smoothing is not necessary here. All this smoothing can be done in a $C^0$-small fashion and thus preserves general position of the resulting foliation with respect to the given polyhedral decomposition. Moreover, as above we can still assume that all ribbons $R_j$ are tangent to the $y$-fibers of the chosen nice coordinates. 

\bigskip

\bigskip

\begin{figure}[h]
\psset{xunit=.5pt,yunit=.5pt,runit=.5pt}

\begin{pspicture}(448.57139888,105.28571437)
\rput(10,80){$\partial^v N(A_i) \longrightarrow$} \rput(480,40){$L_0$}\rput(210,130){$\partial^h N(A_i)$}
{
\newrgbcolor{curcolor}{0 0 0}
\pscustom[linewidth=0.62803688,linecolor=curcolor]
{
\newpath
\moveto(79.81400928,104.97169094)
\lineto(343.04308904,104.97169094)
\lineto(343.04308904,0.31402048)
\lineto(79.81400928,0.31402048)
\closepath
}
}
{
\newrgbcolor{curcolor}{0 0 0}
\pscustom[linewidth=0.45298207,linecolor=curcolor]
{
\newpath
\moveto(113.15504885,87.91633307)
\lineto(312.27347571,87.91633307)
\lineto(312.27347571,15.94075142)
\lineto(113.15504885,15.94075142)
\closepath
}
}
{
\newrgbcolor{curcolor}{0 0 0}
\pscustom[linewidth=4.9999997,linecolor=curcolor]
{
\newpath
\moveto(0,49.0713939)
\lineto(448.57140626,49.0713939)
}
}
{
\newrgbcolor{curcolor}{0 0 0}
\pscustom[linestyle=none,fillstyle=solid,fillcolor=curcolor]
{
\newpath
\moveto(221.09679724,49.30097893)
\curveto(221.09679724,45.75057676)(218.21862803,42.87240713)(214.6682264,42.87240713)
\curveto(211.11782476,42.87240713)(208.23965556,45.75057676)(208.23965556,49.30097893)
\curveto(208.23965556,52.85138109)(211.11782476,55.72955072)(214.6682264,55.72955072)
\curveto(218.21862803,55.72955072)(221.09679724,52.85138109)(221.09679724,49.30097893)
\closepath
}
}
{
\newrgbcolor{curcolor}{0 0 0}
\pscustom[linewidth=0.25300001,linecolor=curcolor]
{
\newpath
\moveto(113.01079626,87.88328159)
\lineto(79.54950726,87.6307416)
}
}
{
\newrgbcolor{curcolor}{0.80000001 0.80000001 0.80000001}
\pscustom[linestyle=none,fillstyle=solid,fillcolor=curcolor]
{
\newpath
\moveto(80.29079721,96.44885108)
\lineto(80.29079721,104.59321059)
\lineto(211.48435039,104.59321059)
\lineto(342.67791257,104.59321059)
\lineto(342.67791257,96.44885108)
\lineto(342.67791257,88.30450156)
\lineto(211.48435039,88.30450156)
\lineto(80.29079721,87.92569159)
\closepath
}
}
{
\newrgbcolor{curcolor}{0 0 0}
\pscustom[linewidth=0.25300001,linecolor=curcolor]
{
\newpath
\moveto(342.90925256,88.00955158)
\lineto(312.22588439,88.00955158)
}
}
{
\newrgbcolor{curcolor}{0 0 0}
\pscustom[linewidth=0.25300001,linecolor=curcolor]
{
\newpath
\moveto(113.01079626,15.88328588)
\lineto(79.54950726,15.6307459)
}
}
{
\newrgbcolor{curcolor}{0 0 0}
\pscustom[linewidth=0.25300001,linecolor=curcolor]
{
\newpath
\moveto(342.90925256,16.13582587)
\lineto(312.35215438,15.88328588)
}
}
{
\newrgbcolor{curcolor}{0.80000001 0.80000001 0.80000001}
\pscustom[linestyle=none,fillstyle=solid,fillcolor=curcolor]
{
\newpath
\moveto(80.32099721,7.84669636)
\lineto(80.32099721,14.71736595)
\lineto(89.80029765,15.04262593)
\curveto(95.01391734,15.22151592)(154.05101382,15.36851591)(220.99384983,15.36928591)
\lineto(342.70810257,15.37028591)
\lineto(342.70810257,8.17294634)
\lineto(342.70810257,0.97560677)
\lineto(211.51455039,0.97560677)
\lineto(80.32099721,0.97560677)
\lineto(80.32099721,7.84628636)
\closepath
}
}
\end{pspicture} \caption{A schematic picture of the neighbourhoods $N(A'_i) \subseteq N(A_i)$ near $\gamma \subseteq L_0$ which is depicted as a large dot. The plane field is smooth on the shaded region}\label{leaf_box}
\end{figure}
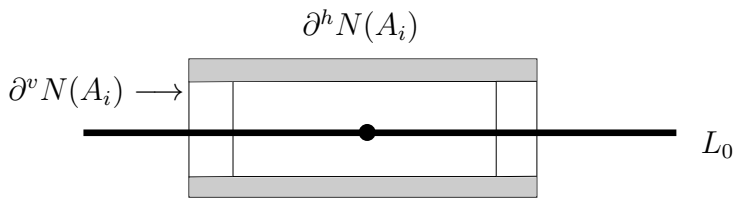

\bigskip

\noindent \textbf{Manufacture and transport contactness:} Apply Lemma \ref{cont_holonomy} to obtain a family of $C^0$-confoliations $(\xi_t)_{t\in[0,\epsilon]}$ that is contact near $\partial ^h N(A_i)$ for $t>0$ and is tangent to $y$-intervals of the chosen very nice coordinates on $N(A_i) \setminus N(A'_i)$ (cf.\ Remark \ref{y_coord}). For $t>0$ the ribbons $R_j$ are contained in the smooth part of the modified $1$-form so that the resulting plane field is a smooth confoliation near each $R_j$ and we can assume that these confoliations are contact near the positive end of each $R_j$ for $t>0$. If the positive end of $R_j$ lies in $\mathcal{O}p(\partial A_i)$, this is true by construction, and if not, then since $R_j$ lies in the smooth region of the foliation we can simply twist around the  $y$-coordinate vector field to achieve this. We then transport the non-integrability along $R_j$ using Lemma \ref{transport_para} to obtain a family of $C^0$-confoliations $(\xi_t)_{t\in[0,\epsilon)}$ outside a neighbourhood of the annular fences. Note that we do not need to alter anything for polyhedra $P$ in $L \cap \mathcal{O}p(\partial^hN(A))$, since by construction $\xi_t$ is contact near $\partial^hN(A)$ for all $t>0$.

\bigskip

\noindent \textbf{Fill in polyhedra:} For any polyhedron $P$ that is disjoint from $N(A'_i)$ each leaf of the characteristic foliation $\xi_t(\partial P)$ passes through the contact region for $t>0$ and hence the holonomy is decreasing away from supporting vertices. Furthermore, it is easy to arrange that $\xi_t$ is a tangential confoliation near $\partial^v N(A_i)$ with respect to $\frac{\partial}{\partial y}$ (cf.\ Definition \ref{tangential_conf}) which is contact on a neighbourhood of the subcomplex $L$ for $t>0$ by simply twisting along the $y$-axis by a small amount after smoothing on annuli $A_i^{\pm}$ parallel (and close) to $\partial^v N(A_i)$ but disjoint from $N(A_i)$. Recall that the foliation -- and also each $\xi_t$ -- has already been made smooth near the $1$-skeleton and after this final modification we can assume that the characteristic foliation induced by $\xi_t$ on $\partial P$ can be coherently integrated away the annuli $A_i^{\pm}$. In order to ensure that the holonomies on $\partial P$ can be defined everywhere in terms of characteristic foliations varying continuously in $t$, it is enough to ensure that the original foliation was smooth near the intersection of $A_i^{\pm}$ and $P^{(2)} \setminus \mathcal{O}p( P^{(1)})$ which consists of a union of transversals contained in faces of $ \partial P$ by general position. But again this provides no serious problems.

In particular, we can assume that $\xi_t$ is contact on the subcomplex $L$ containing $\partial N(A_i)$ in its interior and also that $\xi_t$ is tangent to $\frac{\partial}{\partial y}$ on $L$ as well. We now apply the relative version of Lemma \ref{extend_close} to obtain the desired contact approximation $\xi' $ away from $N(A'_i)$, which is in particular tangent to the $y$-intervals of the chosen nice coordinates near the entire boundary of $N(A_i)$.



 \subsection*{Step 4: Filling in annular holes}
We finally need to fill in a neighbourhood of each annular fence $N(A_i) \cong A \times [-\eta,\eta]$ by a contact structure relative to the boundary. By construction the contact structure is tangent to the $y$-intervals of the chosen very nice coordinates. We then apply Lemma  \ref{movie_quant} to obtain the desired extension, which can be assumed to be $C^0$-close to $\xi'$, since the height $\eta$ of each annular fence was fixed before any of the modifications in the previous step were made and can thus be assumed to be arbitrarily small. This concludes the proof of Theorem \ref{approx_cont}. 

Note that we have formally only constructed a positive contact structure approximating $\mathcal{F}$. To obtain a negative approximating contact structure one simply swaps the orientation of $M$ and applies the same argument.
 
 \section{Consequences and discussion}\label{sec_discussion}
 We now collect some corollaries. The first is the $C^0$-analogue of (\cite{ETh} Corollary 3.2.8) and the proof is identical.
 \begin{cor}\label{taut_tight}
Let $ \mathcal{F}$ be a taut $C^0$-foliation on $M$ that is not the foliation by spheres on $S^2 \times S^1$. Then there are both positive and negative contact structures $\xi_-$ and $\xi_+$ that are symplectically semi-fillable, universally tight and homotopic as plane fields to $T\mathcal{F}$. Moreover, the underlying manifold of the symplectic semi-filling is $M \times [0,1]$.
 \end{cor}
 \begin{proof}
 Let $\mathcal{F}$ be a taut foliation which is not a foliation by planes and let $\omega$ be a dominating closed $2$-form and $\alpha_0$ a continuous defining form for $T\mathcal{F}$. Then $M \times [0,1]$ is a symplectic semi-filling of $(M,\xi_+) \sqcup (-M,\xi_-)$ with symplectic form  $\Omega = \epsilon d(t \thinspace \tilde{\alpha}) + \omega$ for $\epsilon$ sufficiently small and $\tilde{\alpha}$ a smooth approximation of $\alpha_0$. The same is true when one passes to the universal cover and by applying the Gromov-Eliashberg argument to the Bishop family associated to an overtwisted disc it follows that $\xi_+,\xi_-$ are universally tight (cf.\ \cite{ETh}). 
 
 If $\mathcal{F}$ is a foliation by planes then $\mathcal{F}$ is semi-conjugate to a \emph{smooth} foliation by planes, say $\mathcal{F}'$, by Imanishi \cite{Ima}. Hence $T \mathcal{F} \simeq T \mathcal{F}'$ and the result follows from the $C^2$-case. 
 \end{proof}
 \begin{rem}
 Although Corollary \ref{taut_tight} is stated for (everywhere) taut foliations in the strongest sense that there is a closed transversal through every point, in view of the results of \cite{KR3} it also holds for foliations that are only smoothly taut in the sense that every leaf meets a closed transversal.
 \end{rem}
 \noindent A further application of Theorem \ref{approx_cont} is given by the following result, which was noted in \cite{BC}, where the argument was not complete due to the absence of Theorem \ref{approx_cont}.
 \begin{cor}\label{L_Space}
 Let $M$ be a rational homology sphere, which is a Heegaard-Floer homology L-space. Then $M$ admits no taut oriented $C^0$-foliations.
 \end{cor}
 \begin{proof}
By the Corollary \ref{taut_tight}, $M$ admits a contact structure so that $M \times [0,1]$ is a symplectic semi-filling. The result then follows from (\cite{OS}, Theorem 1.4).
 \end{proof}
\noindent Indeed, the argument of Corollary \ref{taut_tight} shows that any contact structure sufficiently $C^0$-close to a taut foliation will be universally tight. In the case of $C^2$-foliations one knows that a positive contact structure approximating a foliation $\mathcal{F}$ is unique up to isotopy, as soon as the obvious necessary conditions are met by results of Vogel \cite{Vog}: i.e.\ $\mathcal{F}$ cannot have torus leaves, nor can it be a foliation by planes or by cylinders only. Vogel's methods use the $C^2$-regularity in an essential way so that it is far from clear whether the corresponding statement should hold for lower regularity. The following example suggests that this should not be the case for general $C^0$-foliations. Note that this example has a minimal set which is a lamination by cylinders (that is without holonomy) and this should perhaps be excluded in a $C^0$-version of Vogel's result.
\begin{ex}\label{fill_lam_surface}
Consider a geodesic lamination $\Lambda$ on a hyperbolic surface $\Sigma$ so that all complementary regions are ideal polygons with an even number of sides. By pulling back under the projection $\Sigma \times S^1 \longrightarrow \Sigma$, this then gives a lamination by cylinders whose complementary regions can be filled by stacks of chairs to obtain a taut foliation. Note that we have a choice of whether each stack of chairs is positively or negatively transverse to the $S^1$-fibers on a given complementary region. By Corollary \ref{taut_tight} this can be approximated by universally tight contact structures. These have been classified by Giroux \cite{Gir} and are determined by a family of vertical tori $\gamma_i \times S^1$ on which the characteristic foliation induced by the contact structure agrees with the one given by the $S^1$-fibers. The curves $\gamma_i$ should approximate $\Lambda$ and it seems unlikely that the curves $\gamma_i$ should be uniquely determined by $\mathcal{F}$ since there are many homotopically distinct curves that approximate $\Lambda$.
\end{ex}
\noindent The above example is quite general:
\begin{ex}
Let $\Lambda$ be an essential lamination whose complementary regions are ideal polygon bundles with an even number of sides. For example let $M_{\varphi}$ be a (hyperbolic) mapping torus with fiber $\Sigma$ and pseudo-Anosov monodromy $\varphi$ and let $\Lambda=\Lambda_{st}$ be the suspension of the stable invariant lamination of $\varphi$. Then the complement of $\Lambda$ can be filled in by stacks of chairs and by choosing the orientations of these chairs appropriately we can assume that the Euler class of the tangent plane field $T \mathcal{F}$ satisfies
$$|e(T\mathcal{F})\cdot [\Sigma]|< 2g(\Sigma) -2.$$
In particular, applying Corollary \ref{taut_tight} we deduce that there are universally tight contact structures on $M_{\varphi}$ such that $e(\xi) \neq \pm(2g(\Sigma) -2) $.
\end{ex}
Note that contact structures on hyperbolic mapping tori with $|e(\xi)\cdot [\Sigma]|= 2g(\Sigma) -2$ were classified by Honda-Kazez-Mati\'{c} \cite{HKM} and, in particular, they are unique up to isotopy. The example above raises the following problem, which is well-understood in the case of Seifert fibered spaces (cf.\ \cite{Bow0}):
\begin{prob}
Classify (universally) tight contact structures on hyperbolic mapping tori. In particular, can they all be obtained as perturbations of the foliations described above?
\end{prob}
More generally, foliations obtained by filling in complements of essential laminations give interesting $C^0$-foliations that will not be of class $C^2$ (e.g.\ they will contradict Sacksteder's Theorem as in Example \ref{fill_lam_surface}). The following example was considered in (\cite{ETh}, p.\ 60) in a slightly different context.
\begin{ex}[Filling in taut sutured complements]\label{fill_lam}
Let $\Lambda$ be any essential lamination on $M$ and suppose that $\Lambda$ is carried by a branched surface $B$. We thicken $B$ to obtain a foliated neighbourbood $N(B)$ of $B$ and note that the complement is naturally a sutured manifold. If it is {\bf taut} in the sense of Gabai \cite{Gab}, then one can fill in the complement of $N(B)$ by a taut foliation to obtain a Reebless foliation on all of $M$.
\end{ex}
Colin \cite{Col} has shown that any Reebless $C^2$-foliation can be $C^0$-approximated by universally tight contact structures. In fact Colin's argument extends readily to the $C^0$-case, once one can approximate Reebless foliations in the sense of Theorem \ref{approx_cont}. Thus we have the following:
 \begin{thm}\label{Reebless}
 Let $\mathcal{F}$ be a Reebless foliation of class $C^0$. Then $T\mathcal{F}$ is homotopic to both positive and negative contact structures that are universally tight.
 \end{thm}
 \noindent Applying Theorem \ref{Reebless} to the foliations described in Example \ref{fill_lam} we obtain universally tight contact structures. Thus the existence of a ``taut'' lamination $\Lambda$ on $M$, in the sense that the complement of a branched surface carrying $\Lambda$ is a taut sutured manifold is sufficient for the existence of (universally) tight contact structures on $M$ and $-M$. 
 
Colin's original result was improved by the author in \cite{Bow} to show that any sufficiently close contact structure to a Reebless foliation of class $C^2$ is in fact universally tight. However, again the arguments of \cite{Bow} use $C^2$-smoothness in an essential way in the form of Kopell's Lemma. Thus it is not obvious if an analogous statement should hold with less regularity and it is false in full generality due to the phenomena of phantom Reeb tori hgihlighted in \cite{KR3}. 

In Section \ref{sec_fol} we gave a working definition of a $C^0$-confoliation as something that naturally occurred in the process of approximating a foliation by a contact structure. Evidently our definition is far too restrictive and the correct definition should characterise those $C^0$-plane fields that can be $C^0$-approximated by contact structures with some additional integrability properties. In particular, this would include a definition of a $C^0$-contact structure and should specialise to the definition of a uniquely integrable plane field in the case that the non-integrable region is empty.

\end{document}